\documentclass[reqno,10pt]{amsart}
\usepackage{amsfonts,amssymb,latexsym,epsfig,amsmath}
\usepackage{amsthm}
\theoremstyle{plain}
\newtheorem{theorem}{Theorem}[section]
\theoremstyle{plain}

\newtheorem{lemma}[theorem]{Lemma}

\theoremstyle{definition}
\newtheorem{definition}{Definition}[section]
\theoremstyle{remark}
\newtheorem{remark}{Remark}[section]

\numberwithin{equation}{section}

\newcommand{\e}{\varepsilon}

\newcommand{\bx}{\bar x}
\newcommand{\by}{\bar y}

\newcommand{\A}{\ensuremath{\mathcal{A}}}

\newcommand{\eps}{\ensuremath{\varepsilon}}
\newcommand{\R}{\ensuremath{\mathbb{R}}}
\newcommand{\N}{\ensuremath{\mathbb{N}}}

\newcommand{\del}{\partial}
\newcommand{\ol}{\overline}
\newcommand{\ra}{\rightarrow}
\newcommand{\alp}{\alpha}

\newcommand{\tr}{\mathrm{tr}}

\begin{document}

\title[Continuous dependence]{Continuous dependence results for Non-linear
  Neumann type boundary value problems}
 
\author[E. R. Jakobsen and C. A. Georgelin]
{Espen R. Jakobsen and Christine A. Georgelin }

\address{Espen R. Jakobsen \newline
    Department of Mathematical Sciences \newline 
    Norwegian University of Science and Technology \newline
    7491 Trondheim, NORWAY }
\email{erj@math.ntnu.no}
\urladdr{http://www.math.ntnu.no/$\sim$erj} 

\address{Christine Georgelin \newline
    Laboratoire de Math\'ematiques et Physique
Th\'eorique Facult\'e des Sciences et Techniques, \newline
F\'ed\'eration Denis Poisson\newline
 Universit\'e de
Tours, Parc de Grandmont \newline
 37200 Tours,  FRANCE  }
\email{christine.georgelin@univ-tours.fr }
\urladdr{http://www.lmpt.univ-tours/$\sim$georgeli}

\date{\today} 

\thanks{Submitted March 26, 2007. accepted July,  2008.}  \thanks {
C. Georgelin was partially supported by the ACI {\sl Mouvements
d'interfaces avec termes non-locaux}, the Department of
Mathematical Sciences at NTNU, and both authors have been supported by
project 176877/V30 of the Research Council of Norway.}

\keywords{Hamilton-Jacobi-Bellman Equations, vanishing viscosity
  method, continuous dependence, boundary value problems, degenerate
  equations, nonlinear PDEs, nonlinear boundary value problems,
  viscosity solutions}

\subjclass[2000]{
35J25, 
35J60, 
35J70, 
49L25. 
}

\begin{abstract}
We obtain estimates on the continuous dependence on the coefficient
for second order non-linear degenerate Neumann type boundary value 
problems. Our results extend previous work of Cockburn
et.al., Jakobsen-Karlsen, and Gripenberg to problems with more general
boundary conditions and domains. 
A new feature here is that we account for the dependence on the
boundary conditions. As one application of our continuous dependence
results, we derive for the first time the rate of convergence for the
vanishing viscosity method for such problems. We also derive new
explicit continuous dependence on the coefficients results for
problems involving Bellman-Isaacs equations and certain quasilinear
equation.
\end{abstract}

\maketitle

\section{Introduction}
\label{sec:intro}

In this paper we will derive continuous dependence estimates for the
following boundary value problem:
\begin{align}
\label{EE}
F(x,u,Du,D^2u)&=0\qquad\text{in}\quad \Omega \quad (\Omega\subset\R^N),\\
G(x,Du)&=0\qquad\text{on}\quad\del\Omega,
\label{BV}
\end{align}
where $u$ is the scalar unknown function, $Du$ and $D^2u$ denote its the gradient
and Hessian, and $\Omega$ is a bounded, smooth ($W^{3,\infty}$) domain in
$\R^N$. Informally speaking, by continuous dependence
estimates we mean estimates of the type 
$$\|u_1-u_2\|\leq \|F_1-F_2\| + \|G_1-G_2\|$$ 
where $u_1$ and $u_2$ are solutions of two different boundary value
problem with data $F_1,G_1$  and $F_2,G_2$. The exact statement is
given in Section \ref{sec:results}.

Equation \eqref{EE} is degenerate elliptic, (possibly)
non-linear, and increasing in $u$. 
This means that the possibly non-linear function $F(x,r,p,X)$ satisfies
$$F(x,r,p,X)\leq F(x,s,p,Y)\quad\text{for all}\quad r\leq s,\quad X\geq Y,$$
where $x \in\Omega$, $r,s\in\R$, $p\in\R^N$, and $X,Y\in \mathbb
S^N$. Here $\mathbb S^N$ is the set of real symmetric $N\times N$ matrices and
$X\geq 0$ in $\mathbb S^N$ means that $X$ is positive semi-definite.
The boundary condition \eqref{BV} 
satisfies the Neumann type condition that $G$ is strictly increasing in $p$ in
 the normal direction:
\begin{align*}
G(x,p+tn(x))\geq G(x,p)+\nu t,
\end{align*}
for some $\nu>0$ and all $t\geq0$, $x\in \del\Omega$, $p$, and outward
unit normal vectors $n(x)$ to $\del\Omega$ at $x$. The
other assumptions on $F$ and $G$ will be specified later.

The class of boundary conditions $G$ we treat in this paper includes
the classical Neumann  condition, $\frac{\del u}{\del n}=g(x)$ in
$\del\Omega$, oblique derivative conditions, and non-linear 
boundary conditions like the capillary condition 
\begin{align*}
&\frac{\del u}{\del n}=\bar\theta(x)(1+|Du|^2)^{1/2}\hspace{-2cm}&
    \text{in}\qquad 
    \del\Omega,\\
\intertext{and the  controlled reflection
condition} 
&\sup_{\alp\in\A}\{\gamma_\alp(x)\cdot Du
-g_\alp(x)\}=0\hspace{-2cm}&\text{in}\qquad \del \Omega.
\end{align*}
In this paper we will assume that $g$, $\bar \theta$, $g_\alp$,
$\gamma_\alp$ are Lipschitz continuous 
functions, that $|\bar \theta|\leq \omega<1$ and $\gamma_\alp\cdot  
n\geq \nu>0$, and that $A$ is a compact metric space.

The main class of equations that our framework can handle are
equations satisfying assumption (H2) in the next section. Loosely
speaking, this is the class of equations where the non-linearity
$F(x,r,p,X)$ is uniformly continuous in $r,p,X$ locally uniformly in $x$.
This case will be referred to as the ``standard case'' in the rest of
this paper. Assumption (H2) excludes most of quasilinear equations,
but contains fully-nonlinear equations like the Bellman-Isaacs
equations from optimal stochastic control and stochastic differential
games theory:
\begin{align*}
\inf_{\theta_1\in \Theta_1}\sup_{\theta_2\in \Theta_2}\left\{-
\tr[a^{\theta_1,\theta_2}(x) D^2u] - b^{\theta_1,\theta_2}(x) Du -
c^{\theta_1,\theta_2}(x) u - f^{\theta_1,\theta_2}(x)\right\}=0
\quad\text{in}\quad \Omega,
\end{align*}
where $\Theta_1,\Theta_2\subset\R^m$ are compact metric spaces, 
$c^{\theta_1,\theta_2}\geq\lambda>0$, the matrices
$a^{\theta_1,\theta_2}\geq0$, and the
coefficients are Lipschitz continuous uniformly in $\theta_1,\theta_2$.
In Sections \ref{sec:ex} and \ref{sec:ext}, we give all the details,
more examples, and extensions to problems on unbounded domains,
time-dependent problems, and certain quasilinear equations like e.g.
\begin{align*}
-\tr\Big[\Big(I-\frac{Du\otimes
    Du}{1+|Du|^2}\Big)D^2u\Big]+\lambda u=f(x)
    \quad\text{in}\quad \Omega.
\end{align*}

Since these equations may be degenerate and non-linear, their solutions
will in general not be smooth. In this paper we work with a concept of
weak solutions called viscosity solutions, a precise definition is
given at the end of this introduction. Viscosity solutions are at
least continuous in the interior of $\Omega$. The boundary conditions will be
interpreted in the weak viscosity sense which essentially means that
either the boundary condition or the equation has to hold on the
boundary. This allow us to have well-posed problems even when the
boundary conditions are classically incompatible. The solutions can
realized by the vanishing viscosity method, and they will be
discontinuous at parts of the boundary where the boundary conditions
are classically incompatible.

An overview of the viscosity solution theory, including Neumann
boundary value problems, can be found in the User's Guide \cite{cil}.
The viscosity solution theory for Neumann type boundary value problems
was initiated by Lions \cite{Li:Neumann} in 1985 for first order
equations, and has been developed by many authors since, see
\cite{pls,I4,ba2,ba3,issa,bl} and references therein for various
aspects of this theory. Today there are two leading  
approaches, one due to Ishii \cite{I4} and another one due to Barles
\cite{ba2,ba3}. They apply under slightly different assumptions and
will be discussed below.

Starting with the standard case, i.e. nonlinear equations
 \eqref{EE} satisfying (H2), we prove under natural and standard
 assumptions, that these boundary value problems have unique H\"older
 continuous viscosity solutions. The H{\"o}lder regularity results are
 new and extend the Lipschitz regularity result of Barles \cite{ba2},
 and we give for the first time a complete proof of such a regularity
 result. We note that these regularity results are global up to the boundary.
 Local up to the boundary H{\"o}lder estimates have previously been
 obtained by Barles-Da Lio \cite{bl} for a different class of
 equations. Whereas our equations are degenererate but strictly
 increasing in the $u$ argument (assumption (H3) in the next section),
 their equations are weakly non-degenerate satisfying some sort of
 ``strong ellipticity condition" but are not necessarily increasing in
 $u$. The arguments needed to prove the two types of results are also
 different, except for some ideas on the construction of test
 functions that are needed in some of the proofs.

Next we prove continuous dependence results comparing  H\"older
continuous (sub and super) solutions of different boundary value problems.
The results we obtain include both continuous dependence on
non-linearities for the equation and the boundary condition. The
results concerning the dependence on the boundary condition are
completely new, at least in a viscosity solution setting, while the results
we obtain for the equations apply to much more general boundary
conditions (including non-linear ones) than earlier results. 

Continuous dependence results for the type of equations
 we consider in
this paper have previously been obtained by e.g. Cockburn
et.al. \cite{CockGripenLonden}, Jakobsen and Karlsen
\cite{JKContDep,JK:Ell,JK:CDIPDE}, and Gripenberg \cite{Gr:CDBC}. 
In all these
papers viscosity solutions methods are used. In some
cases such results can also be obtained from probabilistic arguments,
see e.g. \cite{fs} for results for Bellman equations
set in $\R^N$. Papers \cite{JKContDep,JK:Ell,JK:CDIPDE} treat very
general classes of equations set in $\R^N$ or $\R^N\times[0,T)$,
  \cite{CockGripenLonden} treats zero-Neumann boundary value problem
  for $x$-independent equations, and \cite{Gr:CDBC} treats a
  zero-Dirichlet boundary value problem. 

In the two last papers the domain $\Omega$ is convex and possibly
unbounded and in the last paper further restrictions on the class of
equations are needed (because of the Dirichlet condition) and the
Dirichlet condition is taken in the classical sense. All
these papers treat more general quasilinear equations than we can
treat here, e.g. $p$-Laplace type equations for $p>2$.

The technical explanation for the differences between our continuous
dependence result and the above mentioned results lays in the choice of
test function we use. To 
handle weakly posed Neumann boundary conditions, the idea is to use a
test function that will never satisfy the boundary condition. The
effect will be that the equation holds also at the 
boundary, and that the classical viscosity solution comparison
argument can be used (see the following sections). To achieve this the
usual test function has to be modified and the extent of the
modifications depend on how smooth and non-linear the Neumann 
condition is. To handle possibly non-linear boundary conditions
or H\"older continuous solutions in combination with
boundary reflection directions that are only Lipschitz functions in the
space variable, it seemed that the only available or at least the most
optimal test function to use, is the one constructed by Barles in
\cite{ba2,ba3}.   
As opposed to the basic test function used in the other papers on continuous
dependence, the test function of Barles is not symmetric in its
arguments ($x$ and $y$) and therefore it does not have equal $x$ and
$y$ gradients. We loose a cancellation property in the comparison proof
and hence can not handle as general gradient dependence in the
equations as with the basic test function.
In this paper we consider the same class of non-singular(!) 
equations as Barles in \cite{ba2,ba3}, and this excludes most of the
quasilinear equations considered in
\cite{CockGripenLonden,JKContDep,JK:Ell,Gr:CDBC}, including
$p$-Laplace equations for $p\neq2$ (see also remark
(\ref{noplalpacian}) in section 5.).   

At this point we mentioned that a different test function has
been constructed by Ishii in \cite{I4}. Compared with Barles, Ishii
is able to treat less regular domains but with more regular (and less
non-linear) boundary conditions (e.g. $C^1$ domains and $W^{2,\infty}$
reflections), see \cite{ba2} for a more detailed comparison. 
Using Ishii's test function, continuous dependence results could probably be
obtained under a different set of assumptions (see above). We have not
considered this case.

We also point out that we can handle $u$-depending boundary conditions
only through additional arguments involving transformations. This is in
contrast to the general {\it comparison} results obtained by Barles
\cite{ba3} under similar assumptions for $F$ and $G$. In \cite{ba3}
$u$-dependence is handled  directly by a sort of localization argument
(Lemma 5.2 in \cite{ba3}) 

which only works when you send some parameter 
of the test function to zero. In our continuous dependence arguments, 
we will have to optimize with respect to this parameter and the
optimal choice will in general not be zero or even small. See the
treatment of parameter $\eps$ at end of the proof of Theorem
\ref{main}.
One way to handle $u$-depending Neumann type boundary value problems,
is to transform them into problems with no $u$-dependence, then to use
our results, and finally to transform back.  

We do
not consider such transformations in this paper, instead we refer to
\cite{bl} where such transformations have been considered in a rather
general setting. 

Continuous dependence results have to do with well-posedness of the
equation. Typically the boundary value problem we consider model some
physical process, and the data is measured data. A continuous
dependence result then implies that small measurement errors only
produce small errors in the solutions. Any reasonable model should
satisfy such a requirement in particular in view of numerical computations.
Moreover, continuous dependence results have been used in many other
contexts. They play a key part in the shaking of coefficients approach
of Krylov to obtain error estimates for approximation schemes for 
Bellman equations \cite{Kr:HJB1,Kr:HJB2,Kr:LipCoeff,BJ:Err1,BJ:Err2,BJ:Err3},
in Bourgoing \cite{Bo:C1a} and in \cite{JKContDep} they are used to
obtain regularity results, and they have been used to estimate
diffusion matrix projection errors 
\cite{BOZ}, source term splitting errors \cite{JKR}, and errors
coming from the truncation of Levy measures \cite{JK:CDIPDE}. They
have also been used to derive the rate of convergence for the
vanishing viscosity method
\cite{CockGripenLonden,JKContDep,JK:Ell,Gr:CDBC}, see also 
e.g. \cite{CK06}.

The paper is organized as follows: In the next section we state the
assumptions on the boundary value problem \eqref{EE} and \eqref{BV}
in 
 
the standard case and give
well-posedness and H\"older regularity results. We state the main 
result, the continuous dependence result, and as an immediate
corollary we derive an estimate on the rate of convergence for the
vanishing viscosity method. The proofs of the main result along with
the regularity result are proven in Section \ref{sec:pfs}, and in
Section \ref{sec:ex} we apply our main result to obtain new
continuous dependence results for boundary value problems involving
Bellman-Isaacs equations. We give several extensions of our results in
Section \ref{sec:ext}, to time-depending equations, equations set on
unbounded domains, and certain quasilinear equations. Finally, in
the Appendix we derive the test function used in the
proofs in Section \ref{sec:pfs} along with its properties.

\subsection*{Notations}
We let $|\cdot|$ denote the Euclidean norm both in $\R^m$ (vectors) and
$\R^{m\times p}$ (matrices) for $m,p\in\N$. We denote by
$\mathbb{S}^N$ the space of symmetric $N\times N$ matrices, $\tr$ and
$^T$ denote trace and transpose of matrices, and $\leq$ denote the
natural orderings of both numbers and square matrices. 
For $a,b\in\R$ we define $a\vee b=\max(a,b)$ and $a\wedge
b=\min(a,b)$. We will also denote various constants by $K$ or $C$, and
their values may change from line to line.

Let $BUSC(U)$, $BLSC(U)$, $C(U)$, and $W^{p,\infty}(U)$ denote the spaces of 
bounded upper and lower semicontinuous functions, continuous
functions, and functions with $p$ essentially bounded derivatives, all
functions defined on $U$.
If $f:\mathbb{R}^N\rightarrow \mathbb{R}^{m\times p}$ is a function
and $\alp \in(0,1]$,  
then define the following (semi) norms :
\begin{align*}
   |f|_0=\sup_{x\in \bar{\Omega}}|f(x)|, \qquad
   [f]_{\alp}=\underset{x \neq y}{\sup_{ x,y \in \bar{\Omega}}}
   \frac{|f(x)-f(y)|}{|x-y|^{\alp}},
\qquad \text{and} \qquad
   |f|_{\alp}=|f|_0 + [f]_{\alp}.
\end{align*}
By $C^{0,\alp}(\bar{\Omega})$ we denote the set of functions $f:\to
\mathbb{R}$ with finite norm $|f|_{\alp}$.

We end this section by recalling the definition of a viscosity
solution: 
\begin{definition}
An upper semicontinuous function $u$ is a {\it viscosity subsolution}
of (\ref{EE}) and (\ref{BV}) if for all $\phi\in C^2(\bar{\Omega})$,
at each maximum point $x_0\in\bar{\Omega}$ of $u-\phi$, 
\begin{align}
F(x_0,u(x_0),D\phi(x_0),D^2\phi(x_0))&\le 0\quad\text{if }
x_0\in \Omega,\\ 
\min (
F(x_0,u(x_0),D\phi(x_0),D^2\phi(x_0)),G(x_0,Du(x_0))&\le 0\quad\text{if } x_0\in \partial\Omega
\end{align}
An lower-semicontinuous function $u$ is a {\it viscosity supersolution} of
(\ref{EE}) and (\ref{BV}) if for all $\phi\in C^2(\bar{\Omega})$, at
each minimum point $x_0\in\bar{\Omega}$ of $u-\phi$, 
\begin{align}
F(x_0,u(x_0),D\phi(x_0),D^2\phi(x_0))&\ge 0\quad\text{if }
x_0\in \Omega,\\ 
\max (
F(x_0,u(x_0),D\phi(x_0),D^2\phi(x_0)),G(x_0,Du(x_0))&\ge 0\quad\text{if } x_0\in \partial\Omega
\end{align}
Finally $u$ is a solution when it is both a super and a sub-solution. 
\end{definition}

\section{The main results}
\label{sec:results}

In this section we consider the standard case (when assumption (H2)
below holds).
Following \cite{ba2,ba3} we state the
assumptions on the boundary value problem \eqref{EE} and 
\eqref{BV} and give results on comparison, uniqueness, and existence 
of solutions. Then we give new H\"older regularity results extending the
Lipschitz regularity result of \cite{ba2} in two ways: we allow
H\"older continuous data and small $\lambda$ (see assumption (H3) below). 
We also give a complete proof. The main result of this
paper, the continuous dependence result, is then stated, and as an immediate
consequence we derive an explicit rate for the convergence of the
vanishing viscosity method.

Here is a list of the assumptions we will use, starting by the domain: 
\medskip

\noindent ({\bf H0}) $\ \Omega$ is a bounded
    domain in $\R^N$ with a $W^{3,\infty}$ boundary.
\medskip

\noindent For the equation we use the following standard assumptions: 
\medskip

\noindent({\bf H1}) $\ F \in C(\bar\Omega\times \mathbb{R} \times
       \mathbb{R}^N \times \mathbb{S}^N).$
\medskip

 \noindent ({\bf H2}) \
\begin{minipage}[t]{11cm}
There exists a modulus $\omega_{R,K}$ (a
   continuous, 
   non-decreasing function satisfying $\omega_{R,K}(0)=0$) such that
$$F(y,r,q,Y)-F(x,r,p,X)\leq \omega_{R,K}\Big(|x-y|+\frac
1{\eps^2}|x-y|^2+\eta^2+\eps^2+B\Big),$$
for $\eps,\eta\in(0,1]$, $B\geq0$, $x,y \in\bar\Omega$, $r\in\R$,
$|r|\leq R$, $p,q\in \R^N$ and $X,Y \in \mathbb{S}^N$ satisfying
$|x-y|\le K \eta\e$, $|p-q| \le K (\eta^2+\eps^2+B)$, $|p|+|q|\leq
K(\frac{\eta}{\eps}+\eta^2+\eps^2+B)$,  and
\begin{align}
\label{XY-ineq}
\begin{pmatrix}
              X & 0  \\  
              0 & -Y  
\end{pmatrix}
\le
\ \frac{K}{\e^2}
\begin{pmatrix}
              Id& -Id  \\  
              -Id & Id  
\end{pmatrix}
+K (\eta^2+\eps^2+B)
\begin{pmatrix}
              Id& 0  \\  
              0 & Id  
\end{pmatrix}.
\end{align}
\end{minipage}
\medskip

\noindent ($\overline{\mathrm{\bf H2}}$) \ \begin{minipage}[t]{11cm} There exists
$\alp\in(0,1]$ and $K_R\geq 0$ such that 
$$F(y,r,q,Y)-F(x,r,p,X)\leq
K_R\Big(|x-y|^{\alp}+\frac 1{\eps^2}|x-y|^2+\eta^2+\eps^2+B\Big),$$
where $\eps,\eta,B,R,x,y,p,q,X,Y$ are as in (H2).
\end{minipage}
\medskip

\noindent ({\bf H3}) \ For every $x,p,X,$ and for any $R>0$, there is 
       $\lambda_R > 0$ such that
       $$F(x,r,p,X)- F(x,s,p,X)\geq
       \lambda_{R}(r-s)  \quad\mbox{ for }\quad -R\leq s \leq r \leq R.$$
\smallskip
\noindent The possibly fully nonlinear Neumann type boundary condition
satisfies:
\medskip
   
\noindent ({\bf HB1}) \
\begin{minipage}[t]{11cm} There exists $\nu >0$ such that for all $\mu >0,
x\in \partial\Omega, p\in \R^N,$ 
$$G(x,p+\mu n(x))-G(x,p) \ge \nu \mu,$$
where $n(x)$ is the unit outward normal at $x$.
\end{minipage}
\medskip

\noindent ({\bf HB2}) \ There exists a constant $K$
such that for all $x,y\in\partial \Omega$ and all $p,q\in \R^N,$
$$|G(x,p)-G(y,q)| \le K\left [(1 +|p|+|q|) |x-y| +|p-q|\ \right].$$

\begin{remark}
In general there is a trade off between the regularity of the boundary
$\del\Omega$ and the generality and smoothness of the boundary
condition $G$, see \cite{ba2} for a discussion. (H0) compensates for very 
general non-smooth boundary conditions. 
\end{remark}

\begin{remark}
Assumption (H2) plays the same role as (3.14) in the Users' Guide
\cite{cil}. By this assumption the equation is degenerate
elliptic. Moreover, it is a refined version of assumption (H5-1) in
\cite{ba3} containing also a new parameter $B$. In the proofs, this
parameter will be used to carry information from the boundary
conditions (which are never satisfied, see the introduction) over to
the equations. Assumption ($\ol{\mathrm{H2}}$) is a strengthening of
hypothesis (H2) which yields H\"{o}lder regularity results.

By (H3) the equation is strictly increasing in the $u$ 
argument. Assumption (HB1) is the Neumann assumption, saying that the 
boundary condition $G$, contains non-vanishing and non-tangential (to
$\del\Omega$) oblique derivatives and it is a natural condition to insure
 the well-posedness of the problem.
\end{remark}

We now state a comparison, uniqueness, existence, and regularity
result for solutions of \eqref{EE} and \eqref{BV}. 

\begin{theorem}
\label{WP}
If (H0), (H1), (H2), (H3), (HB1), and (HB2) hold, then the
  following statements are true:
\smallskip

\noindent (a) 
If $u$ is a $BUSC(\bar\Omega)$ subsolution and $v$ is a $BLSC(\bar\Omega)$
  supersolution of \eqref{EE} and \eqref{BV}, then $u\leq v$ in $\bar\Omega$.

\smallskip
\noindent (b) If $\lambda_R$ in (H3) is independent of $R$, then there exists a
  unique solution $u\in C(\bar\Omega)$ of \eqref{EE} and \eqref{BV}.
\smallskip

\noindent (c) Assume ($\ol{H2}$) also holds, $u\in
  C(\bar\Omega)$ is the solution of \eqref{EE} and \eqref{BV}, and 
$\lambda:=\lambda_{|u|_0}>0$. Then there are constants
  $\beta\in(0,\alp]$ and $K$ (only depending on the data and $\lambda$)
  such that
$$|u(x)-u(y)|\leq K|x-y|^{\beta} \quad\text{in}\quad
  \bar\Omega\times\bar\Omega.$$
Furthermore, there exists a constant $\bar\lambda>0$ (only depending
  on the data) such that if $\lambda>\bar\lambda$ then $\beta =\alp$
  (the maximal regularity is attained).
\end{theorem}

The comparison principle in (a) correspond to Theorem 2.1 in
\cite{ba3}. The uniqueness part in (b) follow from (a), and existence
follows from Perrons method \cite{Is:PM} since $w(x):=M-Kd(x)$ is a
supersolution of \eqref{EE}  and $-w$ is a subsolution of \eqref{BV}, 
if $M,K\geq 0$ are big enough,  and $d$ is the $W^{3,\infty}$ extension
of the distance function defined in the Appendix, see
Section 4 in \cite{ba3} for similar results. The regularity result,
part (c), will be proved in Section \ref{sec:pfs}. 
\begin{remark}
The regularity results in part c) are global up to the boundary. Local
 up to the boundary H{\"o}lder estimates  have  been
 obtained by Barles-Da Lio \cite{bl} using different techniques and
 assumptions on the nonlinearity of the equation. See the introduction for a discussion. 
\end{remark}
Now we proceed to the continuous dependence result. We will derive an
upper bound on the difference between a viscosity subsolution $u_1$ of 
\begin{eqnarray}
   \label{E}
   F_1(x,u_1(x),Du_1(x),D^2u_1(x)) &=& 0\quad  \text{ in }\quad
  \Omega, \\
  G_1(x, Du_1(x))&=&0 \quad \text{ on  }\quad  \partial \Omega,\nonumber
\end{eqnarray}
 and a viscosity supersolution $u_2$ of  
\begin{eqnarray}
   \label{E_tmp}
   F_2(x,u_2(x),Du_2(x),D^2u_2(x))& =& 0
  \quad \text{ in }  \quad\Omega,\\
    G_2(x,Du_2(x)) &=& 0\quad \text{ on } \quad \partial
  \Omega.\nonumber 
\end{eqnarray}
We assume the following estimates on the differences of
the two equations and of the two boundary conditions.
\medskip

\noindent {\bf (D1)} There are $\delta_1$, $\delta_2\geq0$,
and $K_F(K)\geq 0$ such that for any $K\geq0$,  
\begin{align*}
  & F_2(y,r,q,Y)  -  F_1(x,r,p,X)
\leq  K_F(K) \Big(\eta ^2+
\delta_1+  \frac{1}{\e^2}\delta_2^2
 +B\Big),
\end{align*}
for $0<\eps\leq\eta:=\eps^{\frac{\bar\alp}{2-\bar\alp}}\leq 1$ with
$\bar\alp=\alp\wedge\beta$, $B\geq0$, 
$x,y \in\bar\Omega$, $r\in\R$, $|r|\leq K$, $p,q\in \R^N$ and
$X,Y \in \mathbb{S}^N$ satisfying $|x-y|\le K \eta\e$, $|p-q| \le K
\eta^2 + K B$, $|p|+|q|\leq
K(\frac{\eta}{\eps}+\eta^2+B)$, and
\begin{align*}
&\begin{pmatrix}
              X & 0  \\  
              0 & -Y  
\end{pmatrix}
\le
\ \frac{K}{\e^2}
\begin{pmatrix}
              Id& -Id  \\  
              -Id & Id  
\end{pmatrix}
+K (\eta^2 +B)
\begin{pmatrix}
              Id& 0  \\  
              0 & Id  
\end{pmatrix}.
\end{align*}

\noindent {\bf (D2)}  There are $\mu_1,\mu_2,K_G\geq 0$ such that for
all $x\in\del\Omega$ and $p\in\R^N$,  
\begin{equation*}
G_2(x,p)- G_1(x,p)\le K_G ( \mu_1 + \mu_2 |p|).
\end{equation*}

\begin{remark}
\label{rem1}
Assumption (D1) is a ``continuous dependence'' version of (H2)
and ($\ol{\mathrm{H2}}$) in this paper, and assumption (3.14) in the
Users' Guide \cite{cil}. A similar assumption is used in Theorem 2.1 in
\cite{JK:Ell}. 

By $\beta$ and $\alp$ we denote the H\"older
exponents of the solutions and data respectively. In general $\alp\geq
\beta$, and equality only holds when $\lambda$ in (H3) is big enough. 

Since $|x-y|\le K\eps\eta$, $\eta=\eps^{\frac{\beta}{2-\beta}}$
imply $\frac{|x-y|^2}{\eps^2}\leq K\eta^2$ and
$|x-y|^{\beta}\leq K\eta^2$, the $F_1-F_2$ inequality in
(D1) will be implied by the following more standard inequality
\begin{align*}
& F_2(y,r,q,Y)  -  F_1(x,r,p,X)
\leq  K\Big(|x-y|^{\alp}+\frac{1}{\e^2} |x-y|^2 +
\delta_1+  \frac{1}{\e^2}\delta_2^2
 +\eta^2+B\Big).
\end{align*}
\end{remark}

\begin{remark} on assumption (D2). In the case of oblique derivative boundary
  conditions, $G_i(x,p)=\gamma_i(x)\cdot p-g_i(x)$, $i=1,2$, and 
$$|G_2(x,p)- G_1(x,p)|\le |(g_1- g_2)^+|_0 + |\gamma_1-\gamma_2| _0
|p|.$$
\end{remark}

Our main result is stated in the following theorem:
\begin{theorem}[Continuous Dependence Estimate] 
\label{main}
Assume (H0), (H1), (H3), (HB1), and (HB2) hold for $H_1,H_2,G_1,G_2$,
  and $u_1,u_2 \in  C^{0,\beta}(\bar{\Omega})$ for $\beta\in(0,1]$. Define
  $\nu^2=(\nu_1\vee\nu_2)(\nu_1\wedge\nu_2)$ and
  $\lambda=\lambda_{1,|u_1|_0}\vee \lambda_{2,|u_2|_0}$.

If (D1) and (D2) hold and $u_1$ and $u_2$ satisfy the boundary value
  problems \eqref{E} and \eqref{E_tmp} respectively, then there exist
  a constant $C>0$ (depending only on $K_F$, $K$, $K_G$
, $|u_1|_\beta$, $|u_2|_\beta$, $\alp$, $\beta$)
Such that  
\begin{equation*}
\lambda \max_{\bar{\Omega}} (u_1-u_2) \ \leq\ C\Big(\delta_1 +
\delta_2^{\alpha\wedge\beta}+\frac{\mu_1}{\nu}+\Big(\frac{\mu_2}{\nu}\Big)^{\alp\wedge\beta}\Big).   
\end{equation*}
\end{theorem}

\begin{remark}
As far as we know this is the first result giving continuous dependence on the
boundary condition. The result also extends the earlier continuous
dependence on the equation type of results of
\cite{CockGripenLonden,JKContDep,JK:Ell,Gr:CDBC} since
much more general boundary conditions are considered (but at the
expense of less general equations!). 
\end{remark}

We prove Theorem \ref{main} in Section \ref{sec:pfs}. An immediate
consequence of this result is an estimate on the rate of
convergence for the vanishing viscosity method.
For $\mu>0$ we consider the solution $u_\mu$ of   
\begin{align}
\label{VV}
F(x,u,Du,D^2u)&=\mu \Delta u \qquad\text{in}\quad \Omega,
\end{align}
with boundary condition \eqref{BV}. The result is the following:

\begin{theorem}
\label{VVthm}
Assume (H0), (H1), ($\ol{H2}$), (H3), (HB1), (HB2),
$\mu>0$, and that $u$ and $u_\mu$ solve
\eqref{EE}/\eqref{BV} and \eqref{VV}/\eqref{BV} respectively. Then $u$
and $u_\mu$ belong to $C^{0,\beta}(\bar\Omega)$ for some
$\beta\in(0,\alp]$ and  
$$|u-u_\mu|_0\leq C\mu^{\beta/2}.$$
\end{theorem}
\begin{proof}
Regularity follows from Theorem \ref{WP}. By assumption ($\ol{\mathrm{H2}}$) 
$$[F(y,r,q,Y)-\mu \,\tr \,Y]-F(x,r,p,X)\leq
C(|x-y|^{\alp}+\frac{|x-y|}{\eps^2}+\eta^2+\eps^2)-\mu\, \tr\, Y,$$
and inequality \eqref{XY-ineq} implies that $-\tr Y\leq C\frac1{\eps^2}+\mathrm{small\ terms}$.
Theorem \ref{main} immediately gives $u-u_\mu\leq
C\mu^{\beta/2}$. A lower bound can be found in a similar way.
\end{proof}

\begin{remark}
This result seems to be the first such result for complicated boundary
condition. We refer to \cite{CockGripenLonden,Gr:CDBC} for results on
weak $0$-Neumann or classical $0$-Dirichlet problems, to \cite{PeSa}
for results on linear Neumann boundary value problems for first order
equations, and to \cite{JKContDep,JK:Ell} for result in $\R^N$ or
$(0,T)\times\R^N$.
\end{remark}

\begin{remark}
The vanishing viscosity method has been studied by many authors
dealing with weak solutions of nonlinear PDEs. The method has 
been used to obtain existence (and uniqueness!) of solutions for
degenerate (e.g. first order) problems by taking the limit as
$\mu\ra0$ (see e.g. \cite{BiBr,soug}), and it is well-known
that it is strongly related to the problem of proving convergence
rates for numerical approximations of such problems (see
e.g. \cite{crli,PeSa}).
\end{remark}

\section{Proofs of Theorems \ref{main} and \ref{WP} (c)}
\label{sec:pfs}
\begin{proof}[Proof of Theorem \ref{main}]
First we assume without loss of generality that
$$\delta_1,\delta_2,\frac{\mu_1}{\nu},\frac{\mu_2}{\nu}\le 1.$$
If this is not the case then the theorem holds since
$$u_1-u_2\leq
(|u_1|_0+|u_2|_0)\Big(\delta_1+\delta_2^{\bar\alp}+\frac{\mu_1}{\nu}+\Big(\frac{\mu_2}{\nu}\Big)^{\bar\alp}\Big),$$
where $\bar\alp=\alp\wedge\beta$.
Then we double the variables and consider
\begin{align*}
 \psi(x,y) &=
   u_1(x)-u_2(y)-\phi(x,y) \quad\text{and}\quad M=\max_{x,y\in\bar\Omega}
   \psi(x,y)=\psi(\bx,\by),
\end{align*}
where for $A,B\geq 0$,
\begin{align*}
\begin{aligned}
   \phi(x,y)&=\frac{1}{\e^2} |x-y|^2 +\frac{A}{\e^2}
   \left(d(x)-d(y)\right)^2 -B(d(x)+d(y))\\
&\quad-\tilde C_2(\frac{x+y}{2},
   \frac{2(x-y)}{\e^2}) (d(x)-d(y)),
\end{aligned}
\end{align*}
and $\tilde C_2(x,p)=C_{2,a}(x,p)$ with
$a=\eta\eps=\eps^{\frac{2}{2-\bar\alp}}$
($\eta=\eps^{\frac{\bar\alp}{2-\bar\alp}}$ by (D1)).
The functions $C_{2,a}$ and $d$ are defined in the Appendix,
and the smooth function $\phi$ was introduced by Barles in
\cite{ba3}. We refer to the Appendix for the proofs 
of the properties of $\phi$.

The existence of a point $(\bar x,\bar y)$ follows from compactness of
$\bar{\Omega}$ and the continuity of all functions involved. 
Since $(\bar x,\bar y)$ is a maximum point,
$$ 2\psi(\bx,\by) \ge \psi(\bx,\bx) +\psi(\by, \by).$$
Moreover, if $A$ is big enough, Lemma \ref{lem_pos} of the Appendix
implies that  
\begin{align}
\label{posA}
\phi(\bx,\by)\geq \frac{1}{2\eps^2}|\bx-\by|^2-K_0\eps^2-B(d(\bx)+d(\by)),
\end{align}
and H\"older regularity of $u_1$ and $u_2$ combined with the last two
inequalities yield 
$$\frac{1}{2\e^2} |\bar x-\bar y|^2\le K_1|\bar x-\bar
y|^{\bar\alp}\vee \eps^2
$$ 
for some constant $K_1$ depending on $K_0$ and the H\"older constants of $u_1$
and $u_2$ (but not on $B$). Equivalently, since
$\eta=\eps^{\frac{\bar\alp}{2-\bar\alp}}$ by (D1) and $\eps\leq\eta$, 
\begin{equation}
\label{bxbyestim}
 |\bx-\by|\le \tilde K_1 \varepsilon^{\frac{2}{2-\bar\alpha}}\,=\tilde
  K_1\eta\eps\quad 
  \hbox{and}\quad \frac{1}{\e^2} |\bx-\by|^2\le \tilde K_1
  \varepsilon^{\frac{2\bar\alpha}{2-\bar\alpha}}\,=\tilde K_1\eta^2.
\end{equation}

Now we choose $A$ and $B$ in the test function $\phi$ to
insure that when $\bx$ or $\by$ 
belong to the boundary $\del\Omega$, then the boundary conditions can
not hold there. See Lemma \ref{lem_BC} of the Appendix. This
means that the {\it equations} always has to hold at $\bx$ and
$\by$. The precise choices of $A$ and $B$ are
$$B = K(\eta^2+\eps^2) + \frac
K{\nu}\Big(\mu_1+\mu_2\frac{\eta}{\eps}\Big)
\quad  \text {and}\quad  A=K,$$ 
for some $K$ only depending on the data of the problem. 

By the maximum principle for semicontinuous functions, Theorem 3.2 of
the "Users' guide" \cite{cil}, there are $(p,X)\in \bar{J}^{2,+}_{\bar \Omega}
u_1(\bx)$ and $(q,Y)\in \bar{J}^{2,-}_{\bar \Omega}u_2(\by)$ such that
\begin{gather*}
p= D_x \phi (\bx,\by), \qquad  q =-D_y\phi(\bx,\by),\\
\begin{pmatrix}
              X & 0  \\  
              0 & -Y  
\end{pmatrix}
\le 
[Id+ \e^2 D^2 \phi (\bx,\by)] D^2 \phi (\bx,\by).
\end{gather*}
Using the definition of viscosity sub and super solutions at $\bx$ and
$\by$ (and Lemma \ref{lem_BC}) we get
\begin{equation*}
F_1(\bx, u_1(\bx),p,X)\le 0\le F_2(\by,u_2(\by),q,Y).
\end{equation*}
We rewrite this as
\begin{equation}
\label{solg1}
F_1(\bx, u_1(\bx),p,X)-F_1(\bx, u_2(\by),p,X)\le
F_2(\by,u_2(\by),q,Y))-F_1(\bx, u_2(\by),p,X). 
\end{equation}
By Lemma \ref{lem_deriv}, the definitions of $p,q,X,Y$, and
$\eps\leq\eta\leq 1$, it follows that
\begin{gather*}
|p-q|\leq K\eta^2+2B, \\
\begin{pmatrix}
              X & 0  \\  
              0 & -Y  
\end{pmatrix}
\le 
\ \frac{K}{\e^2}
\begin{pmatrix}
              Id& -Id  \\  
              -Id & Id  
\end{pmatrix}
+K( \eta^2 +B)
\begin{pmatrix}
              Id& 0  \\  
              0 & Id  
\end{pmatrix},
\end{gather*}
again for some $K$ only depending on the data of the problem.
Since we also have \eqref{bxbyestim}, we are in a position to use
assumption (D1). So if $u_1(\bx)-u_2(\by)\ge 0$, then (D1) and (H3)
applied to \eqref{solg1} yield
\begin{equation*}
\lambda_1 (u_1(\bx)-u_2(\by)) \le  K_F(K) \Big(\eta ^2+
\delta_1+  \frac{1}{\e^2}\delta_2^2
 +B\Big). 
\end{equation*}
By \eqref{posA} and the definition of $\psi$, it follows that
\begin{equation*}
\label{finalstep}
u_1(x)-u_2(x)\le  \psi_\e(x,x)\le 
\psi_\e(\bx,\by) \le u_1(\bx)-u_2(\by) + 2B (d(\bx)+d(\by)).
\end{equation*}
Therefore the two previous inequalities and the choice of $B$ implies that
$$\lambda_1 (u_1(x)-u_2(x)) \le  K \Big(\eta^2 +\delta_1 +
\frac{1}{\e^2}\delta_2^2 +\frac{[\mu_1+ \mu_2
  \frac{\eta}{\e}]}{\nu}      \Big).$$
Remember that $\eta=\eta(\eps)=\eps^{\frac{\bar\alp}{2-\bar\alp}}$ and let
$\eps_1$ and $\eps_2$ be defined by
\begin{align*}
\eta(\eps_1)^2&=\frac1{\eps_1^2}\delta_2^2 &&\text{or}\quad
\eta(\eps_1)^2=\delta_2^{\bar\alp},\\ 
\eta(\eps_2)^2&=\frac{\mu_2\frac{\eta(\eps_2)}{\eps_2}}{\nu} &&\text{or}\quad
\eta(\eps_2)^2=\frac{\mu_2^{\bar\alp}}{\nu^{\bar\alp}} .
\end{align*}
Now with $\eps=\eps_1\vee\eps_2$ ($\leq 1$ by assumption) it
follows that 
\begin{equation*}
\lambda_1 (u_1(x)-u_2(x)) \le  K \Big(\delta_1+  \delta_2^{\bar\alp}
+\frac{\mu_1}{\nu}+\frac{\mu_2^{\bar\alp}}{\nu^{\bar\alp}}\Big). 
\end{equation*}
A closer look at the proof reveals that we
may replace $\lambda_1$ by $\lambda_1\vee\lambda_2$.
\end{proof}

\begin{proof}[Proof of Theorem \ref{WP} (c)]
We start by proving $\alp$-H\"older regularity when $\lambda$ is big
(the last statement of Theorem \ref{WP} (c)). 
The proof is similar to the proof of Theorem \ref{main} except that we have
to modify the test function and use a bootstrap argument. The modified
test function is
\begin{align*}
\begin{aligned}
   \phi_a(x,y)&=\frac{1}{\e^2} e^{-K_e(d(x)+d(y))} |x-y|^2 +\frac{A}{\e^2}
   \left(d(x)-d(y)\right)^2 \\
&\quad-C_a\Big(\frac{x+y}{2},
   \frac{2e^{-K_e(d(x)+d(y))}(x-y)}{\e^2}\Big) (d(x)-d(y))\\
&\quad -K_B(a+\eps^{\frac{2\alpha}{2-\alpha}})(d(x)+d(y)).
\end{aligned}
\end{align*}
We refer to the Appendix for the definitions of $C_a$ and
 $d$.  Playing with the parameter
 $a$, we will use a bootstrap argument to prove that $u$ has the
 right regularity. 

The new test function satisfies similar estimates as the ones
 given in Lemmas \ref{lem_pos} - \ref{lem_deriv}. The moral is that the new
 terms coming from the exponential term are not worse  than the
 old terms. We refer to \cite{ba3} for such
 estimates given in the full generality (but with a different choice
 of $a$).

Now let $\eps\le 1$ and double the variables defining  
$$M:=\psi(\bar x,\bar y)=\sup\psi(x,y)\quad \text{where}\quad
\psi(x,y)=u(x)-u(y)-\phi_a(x,y).$$
If $A$ is big enough, (an easy extension of) Lemma \ref{lem_pos} and
the inequality 
$2\psi(\bx,\by)\geq \psi(\bx,\bx)+\psi(\by,\by)$ imply that
\begin{align}
\label{pospos}
\frac1{2\eps^2}e^{-K_e(d(x)+d(y))}|\bx-\by|^2\leq
2[u(\bx)-u(\by)]+K_0\eps^2
\leq 2|u|_0+K_0.
\end{align}
Define $\eta^2=K^{-1}\frac1{\eps^2}|\bx-\by|^2$ with
$K=e^{2K_e}(2|u|_0+K_0)$. By \eqref{pospos},
$$\eta^2\leq 1\quad \text{and}\quad |\bx-\by|\leq K^{1/2}\eta\eps.$$
We proceed as in the proof of Theorem \ref{main}.

By arguments similar to the ones in the proof of Lemma \ref{lem_BC}, 
if $A$, $K_e$ and $K_B$ are big enough (not depending on $\eps$, $a$
or $B$), then the equation holds even if $(\bar x,\bar y)$ lies on
$\del(\Omega\times\Omega)$. Compared with the proof of Theorem \ref{main}, the
exponential allows us to cancel at the boundary all terms of the form
$\frac{1}{\e^2}|\bx-\by|^2=K\eta^2$ and use
$B=K_B(a+\eps^{\frac{2\alpha}{2-\alpha}})$ at each step.

Note that $D\phi_a$ and $D^2\phi_a$ still satisfy inequalities
\eqref{pmqest} and \eqref{scnd} in Lemma \ref{lem_deriv}. We will choose $a$
such that inequality \eqref{scnd}
takes the form of \eqref{XY-ineq}, i.e. we choose $a$ such that $\frac
{\eps}a \eta^3\leq K$. Since $\eta^2\leq1$ we choose $a=\eps$. Again
we use the definition of  
viscosity solutions and subtract the equations (inequalities) at $\bar
x$ and $\bar y$ using the maximum principle for semi continuous
functions. By the appropriate version of Lemma 
\ref{lem_deriv} and the definition of $\eta^2$ and $B$ we can now use (H1)
and ($\ol{\mathrm{H2}}$) to get 
$$\lambda(u(\bar x)-u(\bar y))\leq 
K\Big(|\bar x-\bar y|^{\alp}+\frac 1{\eps^2}|\bar x-\bar y|^2
+\eta^2+\eps^2+K_B(\eps+\eps^{\frac{2\alpha}{2-\alpha}})\Big).$$  
By Young's inequality, the definition of $\eta^2$, and $\eps\leq1$ we have
$$\lambda(u(\bar x)-u(\bar y))\leq  K\Big(\frac 1{\eps^2} |\bar x-\bar 
y|^2 + \eps^{1\wedge\frac{2\alpha}{2-\alpha}}\Big).$$
When $A$ is big enough, an appropriate version of Lemma \ref{lem_pos},
the definition of $M$, and $0\leq d\leq 1$, imply that 
$$u(\bx)-u(\by)= M+\phi(\bx,\by)\geq M
+\frac1{2\eps^2}e^{-2K_e}|\bx-\by|^2-K_0\eps^2
-K_B(\eps+\eps^{\frac{2\alpha}{2-\alpha}})(d(\bx)+d(\by)).$$  
Combining the two last inequalities and using that $\eps\leq1$ leads to
$$\lambda M \leq
\big(K-\frac{\lambda}2e^{-2K_e}\big)\frac1{\eps^2}|\bar x-\bar
y|^{2}+K\eps^{1\wedge\frac{2\alpha}{2-\alpha}}$$ 
If $\lambda$ is big enough, $\lambda M\leq K
\eps^{1\wedge\frac{2\alpha}{2-\alpha}}$, 
and the definition of $M$ leads to
$$u(x)-u(y) - \phi_{\eps}(x,y)\leq M \leq
\frac{K}{\lambda}\eps^{1\wedge\frac{2\alpha}{2-\alpha}}$$
for every $x,y\in \bar\Omega$. Now by the definition of $\phi_a$, the
properties of the distance function, and Young's inequality, we have 
$$u(x)-u(y) \leq
K\frac{1}{\eps^2}|x-y|^2+K\eps^{1\wedge\frac{2\alpha}{2-\alpha}}.$$ 
If $|x-y|\leq 1$ we may take $\eps=|x-y|^{\frac23}$ when
$1<\frac{2\alp}{2-\alp}$ and
$\eps^{\frac{2\alp}{2-\alp}}=|x-y|^{\alp}$ otherwise, the result
(since we may also interchange $x$ and $y$) is that 
\begin{align}
\label{Hest}
|u(x)-u(y)|\leq K|x-y|^{\frac23\wedge\alp}.
\end{align}
If $|x-y|\geq 1$, the result still holds since then $|u(x)-u(y)|\leq
2|u|_0|x-y|^{\frac23\wedge\alp}$. We are now done if $\alp\leq \frac23$. 

If $\alp\in (\frac23,1]$, we restart the proof using the regularity estimate
\eqref{Hest} to get a better choice of $a$ such that $\frac{\eps}a\eta^3\leq
K$. From \eqref{Hest} and the first inequality in
\eqref{pospos}, 
$$\eta^2=K^{-1}\frac{1}{\eps^2}|\bx-\by|^2\leq K |\bx-\by|^{\frac23}\vee\eps^2 \qquad\text{and
  hence}\qquad \eta\leq K\eps^{\frac{\frac23}{2-\frac23}}\vee\eps,$$  
so the new choice of $a$ should be
$\eps\eta^3=\eps^{\frac52}\vee\eps^4$. But this quantity is
less than $\eps^2$ so we may instead take $a=\eps^2$ which still
implies $\frac{\eps}a\eta^3\leq K$. Now it is a simple exercise to
redo the proof and show that for $\lambda$ big,
\begin{align*}
|u(x)-u(y)|\leq K|x-y|^{\alp} \qquad \text{for}\qquad x,y\in\bar\Omega,
\end{align*}
and this completes the proof of the last part of Theorem \ref{WP}.

Now we will prove the first part of Theorem \ref{WP} (c) using the
result we proved above and an iterative argument of Lions
\cite{Li:Ex}. Here we only sketch parts of the
argument, since the details can be found in \cite{JK:Ell} for similar
equations. The idea is to consider for $\mu>0$
$$F(x,u^{n+1},Du^{n+1},D^2u^{n+1})+\mu u^{n+1} = \mu u^n$$
with boundary conditions \eqref{BV} and noting that $u^n$ converge
uniformly to $u$. If $\mu$ is big enough the above proven 
result applies, and a careful look at the above argument
reveals that when $\lambda+\mu>K(=\bar\lambda)$, then
\begin{align*}
|u^{n+1}(x)-u^{n+1}(y)|\leq \Big(\frac{\mu|u^n|_\alp}{\lambda+\mu - K}+
 \mathrm{other\ terms}\Big)|x-y|^{\alp}, \quad 
 x,y\in\bar\Omega,  |x-y|\leq 1. 
\end{align*}
Furthermore the comparison principle yields
$$|u^{n+1}-u|_0\leq \frac{\mu}{\mu+\lambda}|u^n-u|_0\leq
\Big(\frac{\mu}{\mu+\lambda}\Big)^n|u^0-u|_0.$$ 
When $|x-y|\leq 1$, the rest of the proof is exactly as in
\cite{JK:Ell} and we omit it. When $|x-y|>1$ any H\"older estimate
holds since $u$ is bounded. The result is a H\"older estimate for
any $\lambda>0$, but with a H\"older exponent that is smaller than $\alp$.
\end{proof}

\section{Bellman-Isaacs type boundary value problems}
\label{sec:ex}

In this section we apply our results in Section \ref{sec:results} to
Bellman-Isaacs equations and several different types of boundary conditions.
The Bellman-Isaacs equations are of the form
\begin{align}
\label{HJB}
\inf_{\theta_1\in\Theta_1}\sup_{\theta_2\in\Theta_2}\left\{-
\tr[(\sigma\sigma^T)^{\theta_1,\theta_2}(x) D^2u] -
b^{\theta_1,\theta_2}(x) Du - 
c^{\theta_1,\theta_2}(x) u - f^{\theta_1,\theta_2}(x)\right\}=0
\end{align} 
in $\Omega$. Assumptions (H1), (H2), ($\ol{\mathrm{H2}}$), and (H3)
are satisfied \cite{cil,JK:Ell} if we assume:
\begin{itemize}
\item[1.] $\sigma^{\theta_1,\theta_2}$ and $b^{\theta_1,\theta_2}$ are
Lipschitz continuous in $x$ uniformly in $\theta_1,\theta_2$,
\item[2.] $c^{\theta_1,\theta_2}$ and $f^{\theta_1,\theta_2}$ are $\alp$-H\"{o}lder continuous  in $x$ uniformly in
$\theta_1,\theta_2$,
\item[3.] $c^{\theta_1,\theta_2}(x)\geq\lambda>0$ for all
  $x,\theta_1,\theta_2$, and 
\item[4.] $\Theta_1,\Theta_2$ are  compact metric spaces.
\end{itemize}
Next, we list some typical boundary conditions we can consider:

(a) The classical Neumann condition: 
$$\frac{\del u}{\del n}=g(x) \quad \text{in}\quad \del\Omega.$$ 

(b) The oblique derivative condition:
$$\frac{\del u}{\del \gamma}=g(x)\quad  \text{in}\quad \del\Omega.$$

(c) The capillary boundary condition:
\begin{align}
\label{cap}
\frac{\del u}{\del n}=\bar\theta(x)(1+|Du|^2)^{1/2}\quad \text{in}\quad
    \del\Omega\quad \text{ with } |\bar\theta(x)| \le \omega<1.
\end{align}

(d) The ``controlled'' reflection boundary condition:
\begin{align}
\label{ctrlBC}
\inf_{\alp\in\Theta_1}\sup_{\beta\in\Theta_2}\{\gamma^{\theta_1,\theta_2}(x)
    \cdot Du
    -g^{\theta_1,\theta_2}(x)\}=0\quad \text{in}\quad \del\Omega.
\end{align}
Here $n(x)$ is the outward unit normal to $\del\Omega$. Assumptions (HB1)
and (HB2) hold in all cases if assumption 4 holds along with
\begin{itemize}
\item[5.] there exists $\nu>0$ such that
$$\gamma(x)\cdot
n(x)\geq\nu\qquad\text{and}\qquad \gamma^{\theta_1,\theta_2}(x)\cdot
n(x)\geq\nu\quad\text{uniformly in $\theta_1,\theta_2$},$$
\item[6.] $g,\gamma,\bar
  \theta,\gamma^{\theta_1,\theta_2},g^{\theta_1,\theta_2}$ 
  are Lipschitz continuous in $x$ uniformly  
  in $\theta_1,\theta_2$.
\end{itemize} 

Now we state new continuous dependence results for the for the
Bellman-Isaacs equations \eqref{HJB} combined with the controlled reflection 
boundary conditions \eqref{ctrlBC}:
\begin{theorem}
Assume $u_1$ and $u_2$ satisfy the boundary value problem \eqref{HJB}
and \eqref{ctrlBC} with coefficients
$\sigma_1,b_1,c_1,f_1,\gamma_1,g_1$ and
$\sigma_2,b_2,c_2,f_2,\gamma_2,g_2$ respectively, where both sets of
coefficients satisfy assumptions 1--6 above. 

Then $u_1,u_2$ belong to $C^{0,\beta}(\bar\Omega)$ for some
$\beta\in(0,\alp]$, and
\begin{align*}
  \lambda |u_1-u_2|_0 \leq &\,  C\sup_{\Theta_1 \times\Theta_2}\Big[
   |\sigma_1^{\theta_1,\theta_2}-\sigma_2^{\theta_1,\theta_2}|_0^{\beta} +
   |b_1^{\theta_1,\theta_2}-b_2^{\theta_1,\theta_2}|_0^{\beta} \Big] \\ 
   & + C\sup_{\Theta_1 \times \Theta_2}\Big[
   |c_1^{\theta_1,\theta_2}-c_2^{\theta_1,\theta_2}|_0 +
   |f_1^{\theta_1,\theta_2}-f_2^{\theta_1,\theta_2}|_0\Big] \\ 
&
  +
  \frac{C}{\nu}\sup_{\Theta_1\times\Theta_2}|g^{\theta_1,\theta_2}_1
  -g^{\theta_1,\theta_2}_2|_0+\frac{C}{\nu^{\alp}}   
   \sup_{\Theta_1\times\Theta_2}|\gamma^{\theta_1,\theta_2}_1
  -\gamma^{\theta_1,\theta_2}_2|_0^{\beta}.  
\end{align*}
\end{theorem}
This result is a direct consequence of Theorems \ref{WP} and
\ref{main}. In this case $\delta_1$ correspond to the second line in
the estimate, 
\begin{gather*}
\delta_2^2=C\sup_{\Theta_1 \times
  \Theta_2}[|\sigma^{\theta_1,\theta_2}_1-\sigma^{\theta_1,\theta_2}_2|^2+|b^{\theta_1,\theta_2}_1-b^{\theta_1,\theta_2}_2|^2],\\
 \mu_1=\sup_{\Theta_1 \times
  \Theta_2}|g^{\theta_1,\theta_2}_1-g^{\theta_1,\theta_2}_2|_0,\qquad \mu_2=\sup_{\Theta_1 \times
  \Theta_2}|\gamma^{\theta_1,\theta_2}_1-\gamma^{\theta_1,\theta_2}_2|_0.
\end{gather*}
The dependence on the equation is as in \cite{JKContDep,JK:Ell} and
  the derivation of $\delta_1$ and $\delta_2$ is explained there. 

By Theorem \ref{VVthm} we have for the first time the rate of
  convergence of the vanishing viscosity method for the boundary value
  problem \eqref{HJB} and \eqref{ctrlBC}, i.e.
\begin{align}
\label{mHJB}
\inf_{\theta_1\in\Theta_1}\sup_{\theta_2\in\Theta_2}\left\{-
\tr[(\sigma\sigma^T)^{\theta_1,\theta_2}(x) D^2u] -
b^{\theta_1,\theta_2}(x) Du - 
c^{\theta_1,\theta_2}(x) u -
f^{\theta_1,\theta_2}(x)\right\}=\mu\Delta u
\end{align}
in $\Omega$, with \eqref{ctrlBC} as boundary conditions. The result is
the following:

\begin{theorem}
Assume $u$ and $u_{\mu}$ satisfy \eqref{HJB} and \eqref{mHJB}
respectively with boundary values \eqref{ctrlBC}, and that assumptions
1 -- 6 hold.

Then $u,u_{\mu}$ belong to $C^{0,\beta}(\bar\Omega)$ for some
$\beta\in(0,\alp]$ and  
$$|u-u_\mu|_0\leq C\mu^{\frac{\beta}{2}}.$$
\end{theorem}

\section{Extensions}
\label{sec:ext}

It is possible to consider many kinds of extensions of the results in this
paper. We will consider three cases:  (i)
$\Omega$ unbounded, (ii) time dependent problems, and (iii)
quasilinear equations. In the two first cases the results cover
e.g. Bellman-Isaacs equations under natural assumptions on the data.
\subsection{Unbounded domains}
Let  $\Omega$ be unbounded and let
(H0u) denote assumption (H0) without the boundedness assumption. If we
assume that our sub and supersolutions $u$ and $v$ are bounded, then we
will get continuous dependence and regularity results simply by
following the arguments in this paper replacing the test function
$\phi_a$ by the standard modification
$$\phi_a(x,y)+\gamma(|x|^2+|y|^2),\quad \gamma>0.$$
The new test function will insure existence of maximum points when we
double the variables, and at the end of the proof it turns out (as usual)
that all terms depending on $\gamma$ will vanish when
$\gamma\ra0$. In the proof $B$ will now depend also on the
$\gamma$-terms and the $\gamma$-terms will tend to zero as
$\gamma\ra0$ with a speed depending on $B$, see assumption (D1u)
below. By careful computations and fixing $\eps$ before sending
$\gamma\ra0$ we can conclude as before. We refer to \cite{JK:Ell} for
the details when $\Omega=\R^N$.

The corresponding continuous dependence result will now be given without
further proof. 
We modify assumption (D1) so it corresponds to our new test function, see also
\cite{JK:Ell}: 
\medskip

\noindent {\bf (D1u)} There are $\delta_1$, $\delta_2\geq0$, a modulus
$\omega$, and $K_F(K)\geq 0$,  such that for any $K\geq0$,
\begin{align*}
  & F_2(y,r,q,Y)  -  F_1(x,r,p,X)
\leq  K_F(K) \Big(\eta ^2+
\delta_1+  \frac{1}{\e^2}\delta_2^2
 +B + \gamma(1+|x|^2+|y|^2)\Big), 
\end{align*}
for $\eps,\gamma\in(0,1]$, $\eta:=\eps^{\frac{\alp}{2-\alp}}$, $B\ge 0$, 
$x,y \in\bar\Omega$, $r\in\R$, $|r|\leq K$, $p,q\in \R^N$ and
$X,Y \in \mathbb{S}^N$ satisfying $|x-y|\le K \eta\e$, $|x|+|y|\leq\gamma^{1/2}\omega(\gamma)(1+B)$, $|p-q| \le K(
\eta^2 +  B+\gamma(|x|+|y|))$, $|p|+|q|\leq
K(\frac{\eta}{\eps}+\eta^2+B+\gamma(|x|+|y|))$, 
and 
\begin{align*}
&\begin{pmatrix}
              X & 0  \\  
              0 & -Y  
\end{pmatrix}
\le
\ \frac{K}{\e^2}
\begin{pmatrix}
              Id& -Id  \\  
              -Id & Id  
\end{pmatrix}
+K (\eta^2 +B+ \gamma)
\begin{pmatrix}
              Id& 0  \\  
              0 & Id  
\end{pmatrix}.
\end{align*}

\begin{theorem}[$\Omega$ unbounded] 
\label{ThmUnbnd}
Assume (H0u), (H1), (H3), (HB1), and (HB2) hold for $H_1,H_2,G_1,G_2$,
  and $u_1,u_2 \in  C^{0,\beta}(\bar{\Omega})$ for $\beta\in(0,1]$. Define
  $\nu^2=(\nu_1\vee\nu_2)(\nu_1\wedge\nu_2)$ and
  $\lambda=\lambda_{1,|u_1|_0}\vee \lambda_{2,|u_2|_0}$.

If (D1u) and (D2) hold and $u_1$ and $u_2$ satisfy the boundary value
  problems \eqref{E} and \eqref{E_tmp} respectively then there exist a
  constant $C>0$ (depending only on $K_F$, $K$, $K_G$ 
, $|u_1|_\beta$, $|u_2|_\beta$, $\alp$)
such that  
\begin{equation*}
\lambda \max_{\bar{\Omega}} (u_1-u_2) \ \leq\ C\Big(\delta_1 +
\delta_2^{\alp\wedge\beta}+\frac{\mu_1}{\nu}+\Big(\frac{\mu_2}{\nu}\Big)^{\alp\wedge\beta}\Big).   
\end{equation*}
\end{theorem}

\subsection{Time dependent case}
Consider a Cauchy-Neumann problem of the form:
\begin{align}
\label{tEE}
u_t+F(t,x,u,Du,D^2u)&=0&&\text{in}\quad (0,T)\times\Omega,\\
G(x,Du)&=0&&\text{on}\quad (0,T)\times\del\Omega,\label{tBV}\\
u(0,x)&=u_0(x)&&\text{on}\quad \{0\}\times\Omega.\label{tIV}
\end{align}
In this case we get results by similar arguments as above by replacing
the test function $\phi_a$ by
$$\bar\sigma t+e^{Kt}\phi_a(x,y), \quad \bar\sigma>0.$$
We have to replace assumptions (H1) -- (H3) and (D1) by assumptions
(H1p) -- (H3p) and (D1p) depending on $t$. In (H1p) we assume in addition
continuity in $t$, in (H3p) we allow $\lambda_R\geq0$, and in the last
two assumptions ((H2p) and (D1p)) we simply assume that (H2) and (D1) hold uniformly in
$t$. Note that one can always reduce a problem with  $\lambda_R\in\R$, via an
exponential scaling of $u$, to a problem with $\lambda_R\geq0.$

Now existence, uniqueness, and regularity results follows as before by
appropriately choosing the constants $\bar\sigma$ and $K$. Note
however that in the result corresponding to Theorem \ref{WP} (c) the
H\"older exponent is always $\alp$ and ``maximal regularity'' is
achieved regardless of the value $\lambda$. We refer to
\cite{JKContDep} for such results in the case $\Omega=\R^N$. Now we
state the continuous dependence result without further proof.

\begin{theorem}[Time dependent case] 
\label{ThmTime}
Assume (H0), (H1p), (H3p), (HB1), and (HB2) hold for $H_1,H_2,G_1,G_2$,
$u_1,u_2\in C([0,T]\times\bar\Omega)$, and $u_{1,0},u_{2,0}\in
C^{0,\alp}(\bar\Omega)$ for some $\alp\in(0,1]$. Define
  $\nu^2=(\nu_1\vee\nu_2)(\nu_1\wedge\nu_2)$.

If (D1p) and (D2) hold and $u_1$ and $u_2$ are sub and supersolutions
of initial boundary value problems \eqref{tEE}, \eqref{tBV}, and 
  \eqref{tIV} respectively for $F_1,G_1,u_{1,0}$ and $F_2,G_2,u_{2,0}$,
then there exist a constant $C>0$ (depending only on $K_F$, $K$,
$K_G$, $|u_{1,0}|_\alp$,$|u_{2,0}|_\alp$, $T$, $\alp$)
such that for $t\in(0,T)$,
\begin{equation*}
\max_{\bar{\Omega}} (u_1(t,\cdot)-u_2(t,\cdot))
\ \leq\ |(u_{1,0}-u_{2,0})^+|_0+Ct\Big(\delta_1 + 
\delta_2^{\alpha}+\frac{\mu_1}{\nu}+\Big(\frac{\mu_2}{\nu}\Big)^{\alp}\Big).
\end{equation*}
\end{theorem}
Note that we do not need to assume that $u_1$ and $u_2$ are H\"older
continuous (in $x$) a priori. In fact this regularity follows from the above
theorem! To understand why, and to see details about the derivation in
the case $\Omega=\R^N$, we refer to \cite{JKContDep}.
\subsection{Some quasilinear equations}
\label{sec:extq}
Consider equations of the form
\begin{align}
\label{QLin_2}
-\tr[\sigma(x,Du)\sigma(x,Du)^T D^2u] - f(x,u,Du)+\lambda u =0
\quad\text{in}\quad \Omega,
\end{align}
where $\lambda>0$, $f(x,r,p)$ continuous, increasing in $r$,
$a(x,p)=\sigma(x,p)\sigma(x,p)^T$, and 
\begin{align*}
|\sigma(x,p)-\sigma(y,q)|&\leq K\left(|x-y|+\frac{|p-q|}{1+|p|+|q|}\right),\\
|f(x,r,p)-f(y,r,q)|&\leq K\big[(1+|p|+|q|)|x-y|+|p-q|\big].
\end{align*}
In this case (H1) and (H3) hold in addition to an assumption similar
to (H2). If we also assume (H0), (HB1), and (HB2), then existence and 
comparison for the boundary value problem 
\eqref{QLin_2} and \eqref{BV} was proved in \cite{ba3}. 

More general fully non-linear equations with ``quasilinear'' gradient
dependence can also be considered. We omit this to get a shorter and
clearer presentation. For the same reasons we also restrict ourselves
to the case of Lipschitz continuous solutions and data,
i.e. $\alp=\beta=1,\eta\equiv\eps$.

In this case the quasilinear term in the equation gives rise to a term like 
$$\frac1{\eps^2}|\sigma(p)-\sigma(q)|^2$$
in the proof of the comparison result (when $\sigma$ does not
depend on $x$). By \eqref{pqpq} in Lemma
\ref{lem_deriv}, 
$$|p-q|\leq K|p|\wedge|q||x-y|+K(\eps^2+B)$$
when $\eps$ is small enough, and hence by the assumptions on $\sigma$,
\begin{align}
\nonumber
\frac1{\eps^2}|\sigma(p)-\sigma(q)|^2&\leq \frac{K}{\eps^2}|x-y|^2
+\frac K{\eps^2}\frac {\eps^4+B^2}{1+|p|^2+|q|^2}\\
&\leq \frac 
K{\eps^2}|x-y|^2+K\eps^2+\frac K{\eps^2}B^2.\label{p-q-term}
\end{align}
This computation motivates replacing assumption (D1) by:

\medskip

\noindent {\bf (D1q)} There are $\delta_1$, $\delta_2\geq0$,
and $K_F(K)\geq 0$ such that for any $K\geq0$,  
\begin{align*}
  & F_2(y,r,q,Y)  -  F_1(x,r,p,X)
\leq  K_F(K) \Big(\eta ^2+
\delta_1+  \frac{1}{\e^2}\delta_2^2
 +B +\frac1{\eps^2}B^2\Big),
\end{align*}
for $0<\eps\leq 1$, $B\ge 0$, 
$x,y \in\bar\Omega$, $r\in\R$, $|r|\leq K$, $p,q\in \R^N$ and
$X,Y \in \mathbb{S}^N$ satisfying $|x-y|\le K \e^2$, $|p-q| \le
K|p|\wedge|q|\eps^2 + K(\eps^2 + B),$  $|p|+|q|\leq
K(1+\eps^2+B)$, and
\begin{align*}
&\begin{pmatrix}
              X & 0  \\  
              0 & -Y  
\end{pmatrix}
\le
\ \frac{K}{\e^2}
\begin{pmatrix}
              Id& -Id  \\  
              -Id & Id  
\end{pmatrix}
+K (\e^2 +B)
\begin{pmatrix}
              Id& 0  \\  
              0 & Id  
\end{pmatrix}.
\end{align*}

The continuous dependence result now becomes:

\begin{theorem}[Quasilinear equations, Lipschitz solutions] 
\label{ThmQlin}
Assume (H0), (H1), (H3), (HB1), and (HB2) hold for $H_1,H_2,G_1,G_2$, and
$u_1,u_2\in C^{0,1}(\bar\Omega)$. Define
  $\nu^2=(\nu_1\vee\nu_2)(\nu_1\wedge\nu_2)$.

If (D1q) and (D2) hold and $u_1$ and $u_2$ are sub
and supersolutions of the boundary value problems \eqref{E} and 
  \eqref{E_tmp} respectively, then there exist a constant $C>0$
  (depending only on $K_F$, $K$, $K_G$) such that 
\begin{equation*}
\lambda \max_{\bar{\Omega}} (u_1-u_2) \ \leq\ C\Big(\delta_1 +
\delta_2+\frac{\mu_1}{\nu}+\frac{\mu_2}{\nu}\Big).
\end{equation*}
\end{theorem}

\noindent{\em Proof. }
The proof is similar to the Proof of Theorem \ref{main} with two
exceptions:
\begin{itemize}
\item[(i)] Assume $\delta_1, \delta_2, \frac{\mu_1}{\nu},
  \frac{\mu_2}{\nu}\leq \bar C^{-1}$ where $\bar C$ is big enough (the general
  case follows since $u_1,u_2$ are bounded). Since $\eps$
  is chosen in terms of $\delta_1, \delta_2, \frac{\mu_1}{\nu}, 
  \frac{\mu_2}{\nu}$, a suitable choice of $\bar C$ will ensure that
  $\eps$ is small enough such that \eqref{pqpq} of Lemma
  \ref{lem_deriv} holds. This estimate is needed before one 
  can apply (D1q).
\item[(ii)] At the end of the proof the following estimate will appear
  (remember $\eta=\eps$)
$$\lambda_1 (u_1(x)-u_2(x)) \le  K \Big(\eta^2 +\delta_1 +
\frac{1}{\e^2}\delta_2^2 +\frac{[\mu_1+ \mu_2
  \frac{\eta}{\e}]}{\nu} +\frac1{\eps^2}\Big(\frac{[\mu_1+ \mu_2
  \frac{\eta}{\e}]}{\nu}\Big)^2  \Big),$$
where the {\em new final term} in the right hand side of the inequality
is a consequence of the 
$\frac1{\eps^2}B^2$ term of (D1q). Minimizing $\eps$ like we did in
Theorem \ref{main} then gives the result.  $\hfill\Box$
\end{itemize}

As an example we consider an anisotropic quasilinear equation with
capillary boundary condition. The type of non-linearity appearing here
is similar to the non-linearity appearing in the mean curvature of
graph equation. 
\begin{align*}
-\tr\Big[\sigma\sigma^T\Big(I-\frac{Du\otimes
    Du}{1+|Du|^2}\Big)D^2u\Big]+\lambda u+f(x,u,Du)&=0
    \quad\text{in}\quad \Omega,\\
\frac{\del u}{\del n}-\bar\theta(x)(1+|Du|^2)^{1/2}&=0\quad \text{in}\quad
    \del\Omega,
\end{align*}
where $\bar\theta$ is Lipschitz continuous satisfying $|\bar\theta(x)|
\le \omega<1$ and $f$ satisfies the assumptions mentioned above. 
Assume $u_1$ and $u_2$ are Lipschitz solutions of this boundary value problem
with different $\sigma_1,\sigma_2,\theta_1,\theta_2$ but with same $f$ and
$\lambda$. Then we may apply Theorem \ref{ThmQlin} with
$\delta_1=0=\mu_1$, $\mu_2=|\bar\theta_1-\bar\theta_2|_0$, and 
$$\delta_2^2=\sup_{p\in\R^N}\left|\sqrt{\sigma_1\sigma_1^T\Big(I-\frac{p\otimes
    p}{1+|p|^2}\Big)}-\sqrt{\sigma_2\sigma_2^T\Big(I-\frac{p\otimes
    p}{1+|p|^2}\Big)}\right|^2,$$ 
to obtain
$$\lambda|u_1-u_2|\leq
C\Big(|\sigma_1-\sigma_2|+\frac
{1}{\nu}|\bar\theta_1-\bar\theta_2|_0\Big).$$ 

\begin{remark} 
\label{noplalpacian}
Neither the above assumptions on $\sigma$ nor assumption (D1q) is satisfied
 $p$-Laplacian equations. 

\end{remark}

\appendix
\section{The test function $\phi_a$}
\label{sec:phi}

\subsection{Construction} 
The construction of the test function  follows G. Barles \cite{ba2,ba3} tracking
some quantities more precisely than he needs to. We give
some more details (compared to \cite{ba2,ba3}) whenever we feel this
is helpful for the reader.  

First let $d$ be a $W^{3,\infty}$ extension of the (signed)
distance function from some neighborhood of $\del\Omega$ to $\R^N$ (by (H0) the
distance function is $W^{3,\infty}$ near $\del\Omega$). Furthermore we
may and will choose $d$ such that $0\leq d \leq 1$ and $|Dd(x)|\leq 1$
in $\bar\Omega$. 
We also extend the outward normal vector field $n$ of $\del \Omega$ to
all of $\R^N$ by setting $n(x):=-Dd(x)$ for $x\in\Omega$. An
important consequence of the $W^{3,\infty}$
regularity of $d$ and Taylor's Theorem, are the following inequalities:
\begin{align}
\label{w3}
\pm[d(x)-d(y)]\leq \pm(y-x)\cdot
n\Big(\frac{x+y}2\Big)+\frac1{24}|D^3d|_0|x-y|^3.
\end{align}

Next we extend $G_i, i=1,2$, to a neighborhood $V$ of $\del\Omega$ such
that properties (HB1) and (HB2) still hold here (possibly with
different constants $K$ and $\nu$). Then (HB1) and the intermediate
value theorem that implies that there exists unique solutions
$C_i(x,p)$, of
\begin{align}  
\label{C-def}
G_i(x,p+C_i(x,p) n(x)) = 0,\quad i=1,2,
\end{align}
for every $x$ (in the neighborhood $V$) and $p$. Note that
\begin{align*}
& G_i(x,p)-G_i(x,p+C_i(x,p))= G_i(x,p)=G_i(x,0)+(G_i(x,p)-G_i(x,0))\\ 
\intertext{and}
&G_i(x,p+C_i(y,q) n(x))-G_i(x,p+C_i(x,p) n(x))\\
&= G_i(x,p+C_i(y,q)n(x))=G_i(x,p+C_i(y,q) n(x))-G_i(y,q+C_i(y,q) n(y)),
\end{align*}
so by (HB1) and (HB2)
\begin{align} 
&\nu_i|C_i(x,p)|\leq K(1+|p|)\label{C-bnd}\\
\intertext{and}
\label{C-reg}
&\nu_i |C_i(x,p)-C_i(y,q)|\le K( 1+|p|+|q| )|x-y|+K|p-q|
\end{align}
for $x,y\in V$, $p,q\in\R^N$. 
Now we extend $C_i, i=1,2$, from $V$ to $\R^N$ (in the $x$ variable)
such that \eqref{C-reg} and \eqref{C-bnd} are preserved (possibly with bigger
$K$'s). Next note that  
\begin{align*}
&G_1(x,p+C_1(x,p) n(x))-G_1(x,p+C_2(x,p) n(x))\\
&= -G_1(x,p+C_2(x,p) n(x))=G_2(x,p+C_2(x,p) n(x))-G_1(x,p+C_2(x,p) n(x)).
\end{align*}
Therefore,  for $x\in\del\Omega$, by assumptions $(HB1)$ and (D2) we get
$$\nu_1 (C_1(x,p)- C_2(x,p))\le K_G \Big(\mu_1 +
\mu_2\Big[|p|+K(1+|p|)\Big]\Big).$$ 
We have proved:
\begin{lemma}
\label{propC}
Assume (H0), (HB1), and (HB2).

(a) There exist unique functions $C_1$ and $C_2$ such that 
equation \eqref{C-def} holds in a neighborhood of $\del\Omega$, and
the bounds \eqref{C-bnd} and \eqref{C-reg} hold for all $x,p\in\R^N$.

(b) If in addition (D2) holds, then there exists a constant $K_C>0$
such that for every $p\in\R^N$ and $x\in\del\Omega$, 
\begin{equation}
\label{diffC}
C_1(x,p)- C_2(x,p)\le \frac {K_C}{\nu_1\vee\nu_2}\Big( \mu_1 +
\mu_2(1+ |p|)\Big).
\end{equation}
\end{lemma}

To build smooth functions from $C_1$ and $C_2$ we will use the
following sophisticated regularization due to G. Barles \cite{ba3}:
\begin{definition}
\label{GReg}
Let $a>0$. For a function $C(x,p)$ on $\R^N\times\R^N$ we define
\begin{equation*}
C_a(x,p):=\iint _{\R^N\times \R^N} C(y,q)
\rho\Big((x-y)\frac{\Gamma}{\Lambda}\Big) \rho\Big(
\frac{p-q}{\Lambda}\Big) \frac{\Gamma^N}{\Lambda^{2N}} \,dy\,dq,
\end{equation*}
where 
\begin{equation*}
\Lambda=(a^2 + (p\cdot n(x))^2)^{\frac{1}{2}}\quad \hbox{and}\quad
\Gamma=(1+ |p|^2)^{\frac{1}{2}},
\end{equation*}
and $\rho\in C^{\infty}(\R^N)$ is a nonnegative function with total mass 1
and support in $|x|\leq 1$.
\end{definition}

\begin{lemma}
\label{lemguy}
If $C(x,p)$ satisfies \eqref{C-reg} for $x,p\in\R^N$ and $0<a\le 1$,
then for any $x, p\in \R^N,$
\begin{gather}
\label{Ccoestim}
|C_a (x,p)|\le K \Gamma\quad \text{and}\quad |C_a
 (x,p)-C(x,p)|\le K(a  + |p\cdot n(x)|),\\
\label{Cfoestim}
|D_x C_a(x,p)|\le K \Gamma, \quad |D_p
 C_a(x,p)|\le K,\\
|D_{xx} C_a(x,p)|\le K \frac{\Gamma^2}{\Lambda},
\quad
|D_{xp} C_a(x,p)|\le K \frac{\Gamma}{\Lambda},
\quad
|D_{pp} C_a(x,p)|\le \frac{K}{\Lambda}.
\end{gather}
\end{lemma}

The proof follows from the classical properties of convolution and the regularity
of $C$ \eqref{C-reg} together with the choice of $\Lambda$ and $\Gamma$.

Now remember that $d$ is a $W^{3,\infty}$ extension of the distance
function, and let $C_{2,a}$ be the smooth function obtained
from Definition \ref{GReg} with $C=C_2$. The full test function
takes the following form:
\begin{definition}[The test function $\phi_a$]  
\label{deftestfunct}
\begin{align*}
\begin{aligned}
   \phi_a(x,y)&=\frac{1}{\e^2} |x-y|^2 +\frac{A}{\e^2}
   \left(d(x)-d(y)\right)^2 -B(d(x)+d(y))\\
&\quad-C_{2,a}(\frac{x+y}{2},
   \frac{2(x-y)}{\e^2}) (d(x)-d(y)),
\end{aligned}
\end{align*}
where $A,B\geq 0$ are constants.
\end{definition}

\subsection{Properties} 
In the next 3 lemmas we the state main properties of the test function
$\phi_a$ that we need in this paper.
\begin{lemma}[Positivity]
\label{lem_pos}
Assume (H0), (HB1), (HB2), and let $\phi$ be defined in Definition
\ref{deftestfunct}. If $A$ is big enough (not depending on $\eps,a , B$), then 
\begin{align*}
\phi_a(x,y)\geq \frac{1}{2\eps^2}|x-y|^2-K_0\eps^2-B(d(x)+d(y)).
\end{align*}
\end{lemma}
\begin{proof}
By \eqref{Ccoestim}, $|C_{2,a}(x,p)|\leq C(1+|p|)$ with $C$
independent of $a$. So by Young's inequality one can take $A$ big
enough to insure that
\begin{equation*}
\frac{1}{2\e^2} |x-y|^2 + \frac{A}{2\e^2} \left(d(x)-d(y)\right)^2 -
  C_{2,a}(\frac{x+y}{2}, \frac{2(x-y)}{\e^2}) (d(x)-d(y)) \ge -K_0\eps^2,
\end{equation*}
which proves the Lemma.
\end{proof}

\begin{lemma}[Boundary conditions]
\label{lem_BC}
Assume (H0), (HB1), (HB2), (D2), $0<\eps,\eta, a\leq1$,
and let $\phi_a$ be defined in Definition \ref{deftestfunct}. Then for
any $x,y\in \bar{\Omega}$  such that $|x-y|\le K_1\eta\eps$,
there exists a $K\geq0$ only depending on the data and on $K_1$, such that if 
$$ B = K(\eta^2+\eps^2+a) + \frac
K{\nu_1\vee\nu_2}\Big(\mu_1+\mu_2\frac{\eta}{\eps}\Big)\qquad
\text{and}\qquad  A=K,
$$
then
\begin{equation}
\label{noboundx}
G_1(x,D_x \phi_a(x,y))>0 \quad \text{ if } x\in\partial \Omega,
\end{equation}
\begin{equation}
\label{noboundy}
G_2(x,-D_y \phi_a(x,y))<0 \quad \text{ if } y \in\partial \Omega.
\end{equation}
\end{lemma}

\begin{proof}
We only prove \eqref{noboundx}, the proof of \eqref{noboundy} is
similar but easier due to the choice of $C_{2,a}$ in the test
function. Note that $d(x)=0$, $d(y)-d(x)=|d(x)-d(y)|$, and 
remember that $n=-Dd$. We have 
\begin{align}
\nonumber
D_x\phi_a(x,y) =&\, \frac{2(x-y)}{\e^2} +  \Big[\frac{1}{2}
  D_xC_{2,a}(X,p) +\frac{2}{\e^2} D_pC_{2,a}(X,p)\Big](d(y)-d(x))\\ 
&\,+\Big[C_{2,a}(X,p)+B+\frac{2A}{\e^2}(d(y)-d(x))\Big]
 n(x)\nonumber \\
=& \,p+R_1+[C_1(x,p)+R_2]n(x),\label{phix}
\end{align} 
where $p=\frac1{\eps^2}|x-y|$, $X=\frac12(x+y)$, and
$$R_2:=[-C_1(x,p)+C_2(x,p)]+[-C_2(x,p)+C_{2,a}(X,p)]+B+\frac{2A}{\e^2}(d(y)-d(x)).$$
According to \eqref{Ccoestim}, \eqref{Cfoestim} and \eqref{w3}, the
second term in $R_2$ is bounded below by
$$-K(1+|p|) \frac{|x-y|}{2} - K \big(a + |p\cdot
n(\frac{x+y}2)|\big)\geq
-K\Big(\eta^2+\eps^2+a+\frac{|d(x)-d(y)|}{\eps^2}\Big).$$ 
Here we have also used that $p=\frac2{\eps^2}|x-y|$ and $|x-y|\leq
K\eps\eta$. \\
According to \eqref{diffC}, the first term in $R_2$ is bounded from below by 

\begin{align*}
- \frac {K_C}{\nu_1\vee\nu_2}\Big( \mu_1 + \mu_2(1+ |p|)\Big)
\geq -\frac
K{\nu_1\vee\nu_2}\Big(\mu_1+\mu_2\frac{\eta}{\eps}\Big).
\end{align*}This means that $R_2>0$ if $B> K(\eta^2+\eps^2+a)+\frac
K{\nu_1\vee\nu_2}\Big(\mu_1+\mu_2\frac{\eta}{\eps}\Big)$ and $A\geq
\frac K2.$ 
By \eqref{Cfoestim} of Lemma \ref{lemguy}, the regularity of $d$,
$|x-y|\leq K\eps\eta$, we also find that
\begin{equation*}
|R_1|\leq\Big[K(1+|p| )+ \frac1{\eps^2}K\Big] (d(y)-d(x)) \le
 K\Big(\eta^2 +\eps^2+ \frac{|d(x)-d(y)|}{\e^2}\Big).
\end{equation*}
By (HJB1), (HJB2), \eqref{phix}, \eqref{C-def}, and the above
estimates and choices of $A,B$,
\begin{align*}
&G_1(x,D_x\phi_a(x,y))\ge  G_1(x,p+C_1(x,p)) + \nu_1 R_2 -  K|R_1|\\
&\geq  0+ \nu_1\Big[B-K(\eta^2+\eps^2+a)-\frac
    K{\nu_1\vee\nu_2}\Big(\mu_1+\mu_2\frac{\eta}{\eps}\Big)+(2A-K)\frac{|d(x)-d(y)|}{\eps^2}\Big]\\ 
&\quad-K\Big(\eta^2+\eps^2+\frac{|d(x)-d(y)|}{\eps^2}\Big),
\end{align*}
and the right hand side is strictly positive if 
$$\nu_1 B =
(1+\nu_1)K(\eta^2+\eps^2+a) + \frac
{\nu_1K}{\nu_1\vee\nu_2}\Big(\mu_1+\mu_2\frac{\eta}{\eps}\Big)\quad\text{and}\quad 
2\nu_1A=(1+\nu_1)K.$$ 
\end{proof}

\begin{lemma}[Derivatives]
\label{lem_deriv}
Assume (H0), (HB1), (HB2), and let $x,y\in \bar{\Omega}$ and $\phi_a$
be defined in Definition \ref{deftestfunct}. Then 
\begin{equation}
\label{pmqest}
| D_x \phi_a(x,y) +D_y\phi_a(x,y)|\leq K(\frac {|x-y|^2}{\eps^2}+\eps^2)+2B,
\end{equation}
\begin{equation}
\label{pmqest2}
|D_x \phi_a(x,y)| +|D_y\phi_a(x,y)|\leq
 K(\eps^2+\frac {|x-y|^2}{\eps^2}+\frac {|x-y|}{\eps^2}) + 2B.
\end{equation}
Furthermore if $\frac1{\eps^2}|x-y|^2\leq C$ and $A$ big enough
(independent of $a,\eps,B$) then 
\begin{equation}
\label{lwrbd}
|D_x \phi_a(x,y)| ,|D_y\phi_a(x,y)|\ge
 \frac{|x-y|}{2\eps^2}\left(1-\eps^2\left[B+K(1+A)\right]\right)-K\left(\eps^2+B\right), 
\end{equation}
and if in addition $\eps^2\leq \left[2B+K(1+A)\right]^{-1}$ then
\begin{equation}
\label{pqpq}
|D_x \phi_a(x,y) +D_y\phi_a(x,y)|\leq
 K|D_x \phi_a(x,y)|\wedge|D_y\phi_a(x,y)||x-y| + K(\eps^2+B).
\end{equation}
Finally, if $|x-y|\leq K\eta\eps$ and $0<\eta,\eps,a\leq1$ then
\begin{align}
& D^2 \phi_a (x,y)\nonumber\\
&\le 
\ \frac{K}{\e^2}\big(1+\frac{\eps}{a} \eta^3\big)
\begin{pmatrix}
              Id& -Id  \\  
              -Id & Id  
\end{pmatrix}
+K\Big((1+\frac{\eps}{a} \eta^3)( \eta^2+\eps^2 )+B\Big)
\begin{pmatrix}
              Id& 0  \\  
              0 & Id  
\end{pmatrix}\label{scnd}
\end{align}
\end{lemma}
\begin{proof}
To simplify the computations, we make a change of
variables and define $\Phi_{a,\e}$ by 
$$\phi_a(x,y) = \Phi_{a,\e}(X,Y,Z,T)$$ 
where $\quad X= \frac{x+y}{2},\quad Y =x-y,\quad Z =d(x)-d(y),\quad
T=d(x)+d(y).$ We see that
\begin{equation}
\label{Phidef}
 \Phi_{a,\e} (X,Y,Z,T) = \frac{|Y|^2}{\e^2} - C_{2,a}(X,\frac{2 Y}{\e^2}) Z + \frac{A}{\e^2}Z^2 - BT,
\end{equation}
and straightforward computations lead to 
\begin{align}
\label{primo}
\begin{aligned}
D_X \Phi_{a,\e} (X,Y,Z,T) &= -D_xC_{2,a}(X,\frac{2 Y}{\e^2}) Z, \\ 
D_Y\Phi_{a,\e} (X,Y,Z,T) &= \frac{2 Y}{\e^2}- \frac{2 }{\e^2}D_p 
C_{2,a}(X,\frac{2 Y}{\e^2}) Z, \\
 D_Z \Phi_{a,\e} (X,Y,Z,T) &= -
C_{2,a}(X,\frac{2 Y}{\e^2}) + \frac{2 A}{\e^2}Z,\\ 
D_T \Phi_{a,\e} (X,Y,Z,T)&= -B,
\end{aligned}
\end{align}
and
\begin{align}
D_x \phi_a(x,y)=&\frac{1}{2} D_X \Phi_{a,\e}  +D_Y\Phi_{a,\e} + D_Z
\Phi_{a,\e} Dd(x) +D_T \Phi_{a,\e} Dd(x),\label{DXeqn} \\ 
D_y\phi_a(x,y) =&
\frac{1}{2} D_X \Phi_{a,\e}  -D_Y\Phi_{a,\e} + D_Z \Phi_{a,\e} (-Dd(y))
+D_T \Phi_{a,\e} Dd(y).\label{DYeqn}
\end{align}

First note that estimate \eqref{pqpq} easily follows from
\eqref{pmqest} and \eqref{lwrbd}. We start by proving estimate
\eqref{pmqest}. Using the fact that $n =-Dd$, we see that 
$$ D_x \phi_a(x,y) +D_y\phi_a(x,y) =D_X \Phi_{a,\e} +  D_Z \Phi_{a,\e}
(n(y)-n(x)) + B (n(x)+n(y)).$$
The first term on the right hand side can be estimated by
\eqref{Cfoestim} of Lemma \ref{lemguy},
$$|D_X \Phi_{a,\e} |\le K \Big(1+\frac{|Y|}{\eps^2}\Big) |d(x)-d(y)|,$$
and the second term by \eqref{Ccoestim},
$$\Big| C_{2,a}\Big( X, \frac{2 Y}{\e^2}\Big)\Big|\le K
\Big(1+\frac{|Y|}{\eps^2}\Big)\quad\text{and}\quad  
|D_Z \Phi_{a,\e}|\le K \Big(1+\frac{|Y|}{\eps^2}\Big) + \frac{2
A}{\e^2}|d(x)-d(y)|.$$ 
Combining these estimates, using the regularity of $d$, and $Y=x-y$
then leads to 
\begin{equation*}
|D_x \phi_a(x,y) +D_y\phi_a(x,y)| \le (K+2A|Dd|_0) \Big(1+\frac{|Y
 |}{\eps^2}\Big) |x-y|+ 2|Dd|_0 B, 
\end{equation*}
and estimate \eqref{pmqest} follows 
from Young's inequality. In a similar way, we can also prove \eqref{pmqest2}.

We proceed to prove the lower bound \eqref{lwrbd} and start by
estimating $|D_x \phi_a(x,y) -D_y\phi_a(x,y)|$. Observe that
\begin{gather*}
D_x \phi_a(x,y) -D_y\phi_a(x,y) = 2\big[ \frac{2(x-y)}{\eps^2} - \frac{2}{\eps^2} D_p C_{2,a}(X,\frac{2 Y}{\e^2}) (d(x)-d(y))\big]\\
+(Dd(x)+Dd(y))\big[- C_{2,a}(X,\frac{2 Y}{\e^2}) +
  \frac{2A}{\eps^2}(d(x)-d(y))\big]- B(Dd(x)-Dd(y))
\end{gather*}
Now  using a Taylor expansion and regularity of $d$, we see that
$$(Dd(x)+Dd(y))\cdot(x-y)\leq 2(d(x)-d(y))+ \frac 18|D^3d|_0|x-y|^3,$$
and after applying also \eqref{Ccoestim} we get
\begin{align*}
&(D_x \phi_a(x,y) -D_y\phi_a(x,y))\cdot(x-y)\\
&\geq 4
  \frac{|x-y|^2}{\eps^2}-\frac{4}{\eps^2} D_p C_{2,a}(X,\frac{2
    Y}{\e^2})\cdot(x-y) (d(x)-d(y))\\
&\quad + 2 C_{2,a}(X,\frac{2 Y}{\e^2}) (d(x)-d(y)) 
+\frac{4A}{\eps^2}(d(x)-d(y))^2\\
&\quad
-|D^3d|_0|x-y|^3\left(\frac
A{\eps^2}|Dd|_0|x-y|+K\left(1+\frac{|x-y|}{\eps^2}\right)\right)\\  
&\quad - B (Dd(x) -Dd(y))\cdot(x-y).
\end{align*}
Using Young's inequality as in Lemma \ref{lem_pos} and taking $A$ even
bigger if necessary (but not depending on $a, \eps, B$), we have
($\eps\leq 1$)
\begin{align*}
&(D_x \phi_a(x,y) -D_y\phi_a(x,y))\cdot(x-y)\\
& \ge  \frac{|x-y|^2}{\eps^2}-  B |D d|_0|x-y|^2
-K(1+A)(1+\eps^2)\frac{|x-y|^4}{\eps^2}\\
&= 
\frac{|x-y|^2}{\eps^2}\left(1-\eps^2\left(B+K(1+A)\frac{|x-y|^2}{\eps^2}\right)\right),  
\end{align*}
and Cauchy-Schwarz inequality immediately yields
\begin{equation*}
\label{plowrbd}
|D_x \phi_a(x,y) -D_y\phi_a(x,y))| \ge
 \frac{|x-y|}{\eps^2}\left(1-\eps^2\left(B+K(1+A)\frac{|x-y|^2}{\eps^2}\right)\right). 
\end{equation*}
Now  \eqref{lwrbd} follows by combining the last inequality and \eqref{pmqest},
\begin{align*}
|D_x \phi_a(x,y) |
&\ge
 \frac{|x-y|}{2\eps^2}\left(1-\eps^2\left(B+K+K(1+A)\frac{|x-y|^2}{\eps^2}\right)\right)-K\left(\eps^2+B\right). 
\end{align*}

Now we prove estimate \eqref{scnd}.
A straightforward calculation using (\ref{primo}) yields
\begin{gather*}
D_{XX} \Phi_{a,\e} (X,Y,Z,T) = -D^2_{xx}C_{2,a}(X,\frac{2 Y}{\e^2}) Z ,\\
 D_{XZ} \Phi_{a,\e} (X,Y,Z,T) =-D_{x}C_{2,a}(X,\frac{2 Y}{\e^2}),\\
D_{XY} \Phi_{a,\e} (X,Y,Z,T) = -\frac{2 }{\e^2} D_{xp}C_{2,a}(X,\frac{2
Y}{\e^2}) Z,\\
 D_{YY} \Phi_{a,\e} (X,Y,Z,T) = \frac{2 }{\e^2}Id
-\Big(\frac{2 }{\e^2}\Big)^2 D_{pp}C_{2,a}(X,\frac{2 Y}{\e^2}) Z, \\
D_{YZ} \Phi_{a,\e} (X,Y,Z,T) = -\frac{2 }{\e^2}D_pC_{2,a}(X,\frac{2
Y}{\e^2}) ,\quad D_{ZZ }\Phi_{a,\e} (X,Y,Z,T) = \frac{2 A}{\e^2}Id,\\
D_{TX} \Phi_{a,\e}  = D_{TY} \Phi_{a,\e}  =D_{TZ} \Phi_{a,\e}
 =D_{TT} \Phi_{a,\e} =0.
\end{gather*}
We will estimate these terms using Lemma \ref{lemguy} and
\eqref{w3}.  For example, one has 
\begin{align*}
|  D_{XX} \Phi_{a,\e} (X,Y,Z,T)| &\le
   K\frac{\Gamma^2}{\Lambda}[d(x)-d(y)|\leq K\frac{\Gamma^2}{\Lambda}\Big(\e^2
   |p\cdot n(X)| +
   |D^3d|_0|x-y|^3\Big),
\end{align*}
where $p=\frac {2Y}{\eps^2}$. In our case 
  $\Lambda=[a^2 + (p\cdot n(X))^2]^{\frac{1}{2}}$,
  and $\Gamma=(1+ |p|^2)^{\frac{1}{2}}$, and hence
\begin{align}
| D_{XX} \Phi_{a,\e} (X,Y,Z,T)| &\le K\Big( \e^2\Gamma^2 +
\frac{\Gamma^2}{a}|x-y|^3\Big)\nonumber\\
&\le K\big(\eps^2+\frac {|x-y|^2}{\eps^2}\big)\big(1+\frac{\eps}{a}
\frac{|x-y|^3}{\eps^3}\big).\label{E1} 
\end{align} 
By carefully doing computations like above, one can
prove that
\begin{align}
|D_{YY} \Phi_{a,\e} (X,Y,Z,T)|&\leq \frac{K}{\e^2}\big(1+\frac{\eps}{a}
\frac{|x-y|^3}{\eps^3}\big) ,\label{E2}\\
|D_{YZ} \Phi_{a,\e} (X,Y,Z,T)|&,|D_{ZZ} \Phi_{a,\e}
(X,Y,Z,T) | \leq \frac K{\eps^2}, \label{E2.2}\\
|D_{XZ} \Phi_{a,\e} (X,Y,Z,T)|&\leq K(1+\frac{|x-y|}{\eps^2})\label{E3} \\
|D_{XY} \Phi_{a,\e} (X,Y,Z,T)|&\leq
K(1+\frac{|x-y|}{\eps^2})\big(1+\frac{\eps}{a} \frac{|x-y|^3}{\eps^3}\big). \label{E4} 
\end{align}
Now we compute the matrix $D^2\phi_a (x,y)$  from \eqref{DXeqn} and
\eqref{DYeqn}: 
\begin{align*}
D^2\phi_a (x,y)& =M_1+M_2+ M_3+M4+M_5+M_6 +M_7+ M_8,\\
\intertext{where}
M_1=  &    \begin{pmatrix}
             D^2_{YY}\Phi_{a,\e}& -  D^2_{YY}\Phi_{a,\e}\\  
              -  D^2_{YY}\Phi_{a,\e}&   D^2_{YY}\Phi_{a,\e}
              \end{pmatrix},\\
 M_2= & \,\frac{1}{4} \begin{pmatrix}
             D^2_{XX}\Phi_{a,\e}&  D^2_{XX}\Phi_{a,\e} \\  
               D^2_{XX}\Phi_{a,\e}&  D^2_{XX}\Phi_{a,\e}
              \end{pmatrix},  \\
   M_3= & \begin{pmatrix}
             D^2_{XY}\Phi_{a,\e}& 0\\  
             0&  -D^2_{XY}\Phi_{a,\e}
              \end{pmatrix}, \\          
       M_4=& \,  D^2_{ZZ}\Phi_{a,\e}\begin{pmatrix}
           Dd(x)\otimes Dd(x)&- Dd(x)\otimes Dd(y)\\  
           - Dd(y)\otimes Dd(x)&  Dd(y)\otimes Dd(y)
                         \end{pmatrix},             \\ 
M_5= &\,D^2_{ZY}\Phi_{a,\e}\otimes\begin{pmatrix}  2 Dd(x)&
           - Dd(x)-Dd(y)\\ 
	   - Dd(x)-Dd(y)& 
           2 Dd(y)
              \end{pmatrix}, \\
           M_6= &\, D^2_{ZX}\Phi_{a,\e}\otimes\begin{pmatrix}
              Dd(x)&
              \frac12(Dd(x)-Dd(y))\\   
            \frac{1}{2} ( Dd(x)-Dd(y))&- Dd(y) 
              \end{pmatrix},
              \\
         M_7=& \,    D_Z\Phi_{a,\e}
 \begin{pmatrix}
              D^2d(x)& 0 \\  
              0& -D^2d(y)
\end{pmatrix},
 \\
         M_8=&     -B 
\begin{pmatrix}
              D^2d(x)& 0 \\  
              0& D^2d(y)
\end{pmatrix}.
\end{align*}
It can easily be seen that $M_1$ (use \eqref{E2}), $M_2$ (use
\eqref{E1}), and $M_8$ can be bounded from above by a matrix of the form
\eqref{scnd}. Note that
 $$(M_3(\zeta,\kappa),(\zeta,\kappa)) =(D^2_{XY}\Phi_{a,\e}(\zeta-\kappa),(\zeta+\kappa))\le \frac {1}{\theta^2}
 |D^2_{XY}\Phi_{a,\e}|^2 |\zeta-\kappa|^2 + \theta^2 |\zeta+\kappa|^2,$$
 where $\theta =\eta \sqrt{1+\frac{\eps}{a} \eta^3}$ and hence by \eqref{E3} $M_3$ is also bounded from above by
 \eqref{scnd}. Now we write
 $$M_7=D_Z\Phi_{a,\e} \left(\begin{pmatrix}
              D^2d(x)& 0 \\  
              0& -D^2d(x)
\end{pmatrix}+\begin{pmatrix}
              0& 0 \\  
              0& D^2d(x)-D^2d(y)
\end{pmatrix}\right),$$
and handle the first part of $M_7$ like we did with $M_3$. The second
part can handled using the $W^{3,\infty}$-regularity of $d$ together
with the first order estimates of $\Phi_{a,\e}$. We proceed with $M_4$:
\begin{align*}
(M_4(\zeta,\kappa),(\zeta,\kappa))&=\frac{2A}{\e^2}\Big(\zeta\cdot
 Dd(x)-\kappa\cdot Dd(y)\big)^2\\
&\le 
 \frac{2A}{\e^2} |Dd|_0^2|\zeta-\kappa|^2 + |D^2d|_0^2 \eta^2 (|\zeta|^2+|\kappa|^2).
\end{align*}
The two remaining terms can be treated analogously using also
 \eqref{E2} and \eqref{E3}.
This ends the proof of  the Lemma \ref{lem_deriv}.
\end{proof}

\section{Acknowledgments}
The authors are very grateful to Guy Barles for proposing the problem and
many enlightening discussions along the way.


\begin{thebibliography}{22}

\bibitem{ba2} {G. Barles},
{Fully nonlinear Neumann type boundary conditions for the second order
  elliptic  and parabolic equations } \textit{J. Comp. Differential Equations},
{\bf106}, No. 1, pp 90-106, 1993


\bibitem{ba3}
G. Barles: {\sl Nonlinear Neumann Boundary Conditions for
Quasilinear Degenerate Elliptic Equations and Applications,}
Journal of Diff. Eqns., {\bf 154}, 1999, 191-224.

\bibitem{bl}
G. Barles and F. Da Lio
\newblock Local $C\sp {0,\alpha}$ estimates for viscosity solutions of
Neumann-type boundary value problems. 
\newblock {\em  J. Differential Equations} 225 (2006), no. 1, 202--241.

\bibitem{BJ:Err1}
G. Barles and E. R. Jakobsen. 
\newblock On the convergence rate of approximation schemes for
Hamilton-Jacobi-Bellman equations. 
\newblock {\em M2AN Math. Model. Numer. Anal.} 36(1): 33-54, 2002.

\bibitem{BJ:Err2}
G. Barles and E. R. Jakobsen. 
\newblock Error bounds for monotone approximation schemes for
Hamilton-Jacobi-Bellman equations. 
\newblock {\em SIAM J. Numer. Anal.} 43(2):540-558, 2005.

\bibitem{BJ:Err3}
G. Barles and E. R. Jakobsen. 
\newblock
Error bounds for monotone approximation schemes for parabolic
Hamilton-Jacobi-Bellman equations.
\newblock
To appear in {\em Math. Comp.} 

\bibitem{BOZ}
F.~Bonnans, E.~Ottenwaelter, and H.~Zidani.
\newblock A fast algorithm for the two dimensional HJB equation of
stochastic control.
\newblock  {\em M2AN Math. Model. Numer. Anal.}  38(4):723--735, 2004.

\bibitem{Bo:C1a}
M.~ Bourgoing.
\newblock $C\sp {1,\beta}$ regularity of viscosity solutions via a
  continuous-dependence result. 
\newblock {\em Adv. Differential Equations} 9 (2004), no. 3-4, 447--480.

\bibitem{BiBr}
S. Bianchini and A. Bressan.
\newblock Vanishing viscosity solutions of nonlinear hyperbolic
systems. 
\newblock {\it Ann. of Math.} (2) 161 (2005), no. 1, 223--342.


\bibitem{CK06}
G.-Q. Chen and K. H. Karlsen. 
\newblock
$L\sp 1$-framework for continuous dependence and error estimates for
quasilinear anisotropic degenerate parabolic equations.  
\newblock {\em Trans. Amer. Math. Soc.} 358 (2006), no. 3, 937--963.


\bibitem{CockGripenLonden}
B.~Cockburn, G.~Gripenberg, and S.-O. Londen.
\newblock Continuous dependence on the nonlinearity of viscosity solutions of
  parabolic equations.
\newblock {\em J. Differential Equations}, 170(1):180--187, 2001.

\bibitem{cil}
M.G Crandall, H.Ishii and P.L Lions: {\sl User's guIde to
viscosity solutions of second order Partial differential
equations.} Bull. Amer. Soc. {\bf 27} (1992), pp 1-67.

\bibitem{crli}
M. G. Crandall and P.-L. Lions. 
\newblock Two approximations of solutions of Hamilton-Jacobi
equations.
\newblock {\em Math. Comp.} 43 (1984), no. 167, 1--19. 


\bibitem{fs}
W.H Fleming and H.M Soner: {\sc controlled markov processes and
viscosity solutions.} Applications of Mathematics,
Springer-Verlag, New-York, 1993.

\bibitem{gs} Y.\ Giga and M.-H.\ Sato : {\sl Generalized interface
evolution with Neumann boundary condition}, Proc.\ Japan Acad.
{\bf 67 Ser.\ A} (1991), 263--266.

\bibitem{gt}
D. Gilbarg and N.S. Trudinger: {\sc Elliptic Partial Differential
Equations of Second-Order.} Springer, New-York, (1983).


\bibitem{Gr:CDBC}
G.~Gripenberg. 
\newblock Estimates for viscosity solutions of parabolic equations
with Dirichlet boundary conditions. 
\newblock {\em Proc. Amer. Math. Soc.} 130 (2002), no. 12, 3651--3660. 

\bibitem{I4} 
H. Ishii 
\newblock 
Fully nonlinear oblique derivative
problems for nonlinear second-order elliptic PDE's. 
\newblock {\sl  Duke Math. J.} {\bf 62} (1991), pp~663-691.

\bibitem{Is:PM} H. Ishii.
\newblock Perron's method for Hamilton-Jacobi
 equations.
\newblock {\em Duke Math. J.} {\bf 55} (1987), pp~369--384.

\bibitem{issa} H. Ishii and M.-H. Sato
\newblock Nonlinear oblique derivative problems for singular
degenerate parabolic equations on a general domain. 
\newblock {\em Nonlinear Anal.} 57 (2004), no. 7-8, 1077--1098.

\bibitem{JKContDep}
E.~R. Jakobsen and K.~H. Karlsen.
\newblock Continuous dependence estimates for viscosity solutions of
fully nonlinear degenerate parabolic equations. 
\newblock {\em J. Differential Equations} 183:497-525, 2002.

 \bibitem{JK:Ell}
 E.~R. Jakobsen and K.~H. Karlsen.
 \newblock Continuous dependence estimates for viscosity solutions of
 fully nonlinear degenerate elliptic equations. 
\newblock {\em Electron. J. Diff. Eqns.} 2002(39):1--10, 2002. 

 \bibitem{JK:CDIPDE}
E. R. Jakobsen and K. H. Karlsen. 
\newblock Continuous dependence estimates for viscosity solutions of
integro-PDEs.
\newblock {\em J. Differential Equations} 212(2): 278-318, 2005.

\bibitem{JKR}
E. R. Jakobsen, K. H. Karlsen, and N. H. Risebro. 
\newblock On the convergence rate of operator splitting for
Hamilton-Jacobi equations with source terms.  
\newblock {\em Siam J. Numer. Anal.} 39(2): 499-518, 2001. 

 \bibitem{Kr:HJB1}
 N.~V. Krylov.
 \newblock On the rate of convergence of finite-difference approximations
 for Bellman's equations.
 \newblock {\em St. Petersburg Math. J.}, 9(3):639--650, 1997.

\bibitem{Kr:HJB2}
N.~V. Krylov.
\newblock On the rate of convergence of finite-difference approximations
for Bellman's equations with variable coefficients.
\newblock {\em Probab. Theory Ralat. Fields}, 117:1--16, 2000.

\bibitem{Kr:LipCoeff}
N.~V. Krylov.
\newblock On the rate of convergence of finite-difference
approximations for Bellman equations with Lipschitz coefficients. 
\newblock {\em Appl. Math. Optim.} 52 (2005), no. 3, 365--399.

\bibitem{Li:Ex}
P.-L. Lions
\newblock Existence results for first-order Hamilton-Jacobi
equations. 
\newblock {\em Ricerche Mat.} 32 (1983), no. 1, 3--23.

\bibitem{Li:Neumann}
P.-L. Lions
\newblock Neumann type boundary conditions for Hamilton-Jacobi equations. 
\newblock {\em Duke Math. J.} 52 (1985), no. 4, 793--820.

\bibitem{pls} Lions, P.L. and Sznitman A.S.
{\sl  Stochastic Differential Equations with reflectiong Boundary conditions} {\it Com. on Pure and Applied Mathematics} {\bf 37},  (1984), No.1, 511-537

\bibitem{PeSa}
B. Perthame and R. Sanders.
\newblock The Neumann problem for nonlinear second order singular
perturbation problems. 
\newblock {\em SIAM J. Math. Anal.} 19 (1988), no. 2, 295--311.


\bibitem{soug}
P. E. Souganidis
\newblock Existence of viscosity solutions of Hamilton-Jacobi
equations.
\newblock {\it J. Differential Equations} 56 (1985), no. 3, 345--390. 

\end{thebibliography}
\end{document}